\documentclass[12pt,reqno]{amsart}
\usepackage{txfonts}
\usepackage{amssymb}
\usepackage{amsmath}
\allowdisplaybreaks
\usepackage{enumitem}
\usepackage{xcolor}
\usepackage{eucal}
\usepackage{mathtools}
\usepackage{graphicx}
\usepackage{xcolor}
\usepackage{amsmath}
\usepackage{amscd}
\usepackage[all]{xy}           
\usepackage{tikz}
\usepackage{amsfonts,latexsym}
\usepackage{xspace}
\usepackage{epsfig}
\usepackage{float}
\usepackage{mathrsfs} 
\usepackage{color}
\usepackage{fancybox}
\usepackage{colordvi}
\usepackage{tikz}
\usetikzlibrary{positioning}
\usepackage{multicol}
\usepackage{tikz-cd}
\usepackage{colordvi}
\usepackage{ifpdf}
\usepackage{amsmath, amssymb, amsthm, mathrsfs}
\ifpdf
  \usepackage[colorlinks,final,backref=page,hyperindex]{hyperref}
\else
  \usepackage[colorlinks,final,backref=page,hyperindex,hypertex]{hyperref}
\fi
\usepackage[active]{srcltx} 

\usepackage{graphicx}
\usepackage{extarrows}

\newtheorem{theorem}{Theorem}[section]
\newtheorem{prop}[theorem]{Proposition}
\newtheorem{lemma}[theorem]{Lemma}
\newtheorem{coro}[theorem]{Corollary}
\newtheorem{prop-def}{Proposition-Definition}[section]
\newtheorem{coro-def}{Corollary-Definition}[section]

\theoremstyle{definition}
\newtheorem{defn}[theorem]{Definition}
\newtheorem{remark}[theorem]{Remark}

\newtheorem{exam}[theorem]{Example}

\newcommand{\nc}{\newcommand}
\nc{\tred}[1]{\textcolor{red}{#1}}
\nc{\tblue}[1]{\textcolor{blue}{#1}}
\nc{\tgreen}[1]{\textcolor{green}{#1}}
\nc{\tpurple}[1]{\textcolor{purple}{#1}}
\nc{\btred}[1]{\textcolor{red}{\bf #1}}
\nc{\btblue}[1]{\textcolor{blue}{\bf #1}}
\nc{\btgreen}[1]{\textcolor{green}{\bf #1}}
\nc{\btpurple}[1]{\textcolor{purple}{\bf #1}}
\nc{\NN}{{\mathbb N}}
\nc{\ncsha}{{\mbox{\cyr X}^{\mathrm NC}}} \nc{\ncshao}{{\mbox{\cyr
X}^{\mathrm NC}_0}}

\newcommand{\delete}[1]{}

\nc{\mlabel}[1]{\label{#1}}
\nc{\mcite}[1]{\cite{#1}}
\nc{\mref}[1]{\ref{#1}}
\nc{\meqref}[1]{\eqref{#1}}
\nc{\mbibitem}[1]{\bibitem{#1}}

\delete{
\nc{\mlabel}[1]{\label{#1}{\hfill \hspace{1cm}{\bf{{\ }\hfill(#1)}}}}
\nc{\mcite}[1]{\cite{#1}{{\bf{{\ }(#1)}}}}
\nc{\mref}[1]{\ref{#1}{{\bf{{\ }(#1)}}}}
\nc{\meqref}[1]{\eqref{#1}{{\bf{{\ }(#1)}}}}
\nc{\mbibitem}[1]{\bibitem[\bf #1]{#1}}
}
\font\cyr=wncyr10 
\font\scyr=wncyr6
\nc{\sha}{{\mbox{\scyr X}}}
\nc{\shap}{{\mbox{\cyrs X}}} 
\nc{\shpr}{\diamond}    
\nc{\shp}{\ast} \nc{\shplus}{\shpr^+}
\nc{\shprc}{\shpr_c}    
\nc{\dep}{\mrm{dep}} \nc{\lc}{\lfloor} \nc{\rc}{\rfloor}
\nc{\db}{\leq_{\rm db}} \nc{\bfk}{\bf k}

\nc{\cala}{{\mathcal A}} \nc{\calb}{{\mathcal B}}
\nc{\calc}{{\mathcal C}}
\nc{\cald}{{\mathcal D}} \nc{\cale}{{\mathcal E}}
\nc{\calf}{{\mathcal F}} \nc{\calg}{{\mathcal G}}
\nc{\calh}{{\mathcal H}} \nc{\cali}{{\mathcal I}}
\nc{\call}{{\mathcal L}} \nc{\calm}{{\mathcal M}}
\nc{\caln}{{\mathcal N}} \nc{\calo}{{\mathcal O}}
\nc{\calp}{{\mathcal P}} \nc{\calr}{{\mathcal R}}
\nc{\cals}{{\mathcal S}} \nc{\calt}{{\mathcal T}}
\nc{\calu}{{\mathcal U}} \nc{\calw}{{\mathcal W}} \nc{\calk}{{\mathcal K}}
\nc{\calx}{{\mathcal X}} \nc{\CA}{\mathcal{A}}

\nc{\fraka}{{\mathfrak a}} \nc{\frakA}{{\mathfrak A}}
\nc{\frakb}{{\mathfrak b}} \nc{\frakB}{{\mathfrak B}}
\nc{\frakc}{{\mathfrak c}}
\nc{\frakD}{{\mathfrak D}} \nc{\frakF}{\mathfrak{F}}
\nc{\frakf}{{\mathfrak f}} \nc{\frakg}{{\mathfrak g}}
\nc{\frakH}{{\mathfrak H}} \nc{\frakL}{{\mathfrak L}}
\nc{\frakM}{{\mathfrak M}} \nc{\bfrakM}{\overline{\frakM}}
\nc{\frakm}{{\mathfrak m}} \nc{\frakP}{{\mathfrak P}}
\nc{\frakN}{{\mathfrak N}} \nc{\frakp}{{\mathfrak p}}
\nc{\frakS}{{\mathfrak S}} \nc{\frakT}{\mathfrak{T}}
\nc{\frakX}{{\mathfrak X}}

\nc{\RR}{\mathbb{R}}  \nc{\QQ}{\mathbb{Q}} \nc{\ZZ}{\mathbb{Z}}
\nc{\lbar}{\overline}

\font\cyr=wncyr10 \font\cyrs=wncyr7
\nc{\li}[1]{\textcolor{red}{#1}}
\nc{\lir}[1]{\textcolor{red}{Li:#1}}
\nc{\yi}[1]{\textcolor{blue}{Yi: #1}}
\nc{\xing}[1]{\textcolor{purple}{Xing:#1}}
\nc{\revise}[1]{\textcolor{red}{#1}}
\nc{\lin}[1]{\textcolor{blue}{GLin:#1}}

\begin{document}

\title[Averaging pre-Lie bialgebras]{Averaging pre-Lie bialgebras}
%

\author{Lin Gao
}
\address{School of Mathematics and Statistics, Henan University, Henan, Kaifeng 475004, P.\,R. China}
\email{gaolin821000@163.com
}

\author{Mengke Yang
}
\address{School of Mathematics and Statistics, Henan University, Henan, Kaifeng 475004, P.\,R. China}
\email{1878646898@qq.com
}

\author{Yuanyuan Zhang$^{*}$
}
\footnotetext{* Corresponding author.}
\address{School of Mathematics and Statistics, Henan University, Henan, Kaifeng 475004, P.\,R. China}
\email{zhangyy17@henu.edu.cn
}

\date{\today}
\begin{abstract}
In this paper, we first introduce representations of averaging pre-Lie algebras and study their matched pairs, Manin triples, and bialgebra theories. We prove that these three notions are equivalent under certain conditions. Moreover, by introducing averaging operators on quadratic Rota-Baxter pre-Lie algebras, we show that such operators give rise to averaging pre-Lie bialgebras. Then we introduce the notion of admissible classical Yang-Baxter equations in averaging pre-Lie algebras, as well as the relative Rota-Baxter operators on averaging pre-Lie algebras, and show that the relative Rota-Baxter operators on averaging pre-Lie algebras yield symmetric solutions of admissible classical Yang-Baxter equations in averaging pre-Lie algebras. Finally, we generalize the concept of averaging Lie bialgebra introduced by Hou in~\cite[Definition~4.4]{HSZh26} (J. Algebra) and show that every averaging pre-Lie bialgebra induces an averaging Lie bialgebra within our framework.
\end{abstract}

\makeatletter
\@namedef{subjclassname@2020}{\textup{2020} Mathematics Subject Classification}
\makeatother
\subjclass[2020]{
	16W99, 
	17B38,  
	16T10} 

\keywords{averaging pre-Lie algebra, averaging pre-Lie bialgebra, classical Yang-Baxter equation, Manin triple, relative Rota-Baxter operator}

\maketitle

\tableofcontents

\setcounter{section}{0}



\section{Introduction}

The purpose of this paper is to develop a bialgebra theory for averaging pre-Lie algebras and construct averaging pre-Lie bialgebras by employing relative Rota-Baxter operators and symmetric solutions of the classical Yang-Baxter equation in averaging pre-Lie algebras.

\vspace{0.5cm}
{\noindent\bf Averaging operators.}
Let $A$ be a commutative topological algebra over the field of real numbers $\mathbb{R}$, an averaging operator $P$ is a linear continuous operator on $A$ satisfying
\[P(xP(y)) = P(x)P(y).
\]
This family of operators was initially applied implicitly by O. Reynolds~\cite{Rey41} in turbulence theory research. An important class of averaging operators used in turbulence theory is the class of averages over one portion of space-time of certain vector fields.
For example,
\[
\bar{f}(x, t) = \lim_{T \to \infty} \frac{1}{2T} \int_{-T}^{T} f(x, t + \tau) \, d\tau,
\]
the time average of a real function \( f \) defined on space-time is such an operator.
In case \( f(x, \tau) \) is an integrable function of \( \tau \) which is bounded above or below by a constant, then \( \bar{f}(x, t) \) does not depend on \( t \).
Since then, averaging operators have attracted extensive attention in the fields of algebra and mathematical physics, and have been intensively studied in various scenarios. Cao~\cite{Cao00} studied averaging operators from the perspective of general algebra. Subsequently, Aguiar~\cite{Agu00} has investigated averaging operators on various specific types of algebras, such as associative algebras, Lie algebras and Leibniz algebras. Further developments on averaging algebras can be found in~\cite{B62, GZ18, PG15, M96, ZhG25}.

\vspace{0.5cm}
{\noindent\bf Pre-Lie bialgebras.}
Duality between algebras and coalgebras is a fundamental theme in modern algebra. By endowing a single vector space with both an algebra structure and a coalgebra structure, imposing appropriate compatibility conditions between the multiplication and comultiplication, one obtains the notion of a bialgebra. Such structures play a significant role in the theory of quantum groups, integrable systems, and mathematical physics. In~\cite{CP94,Dr83}, authors investigate Lie bialgebras, which constitute the algebraic structures underlying Poisson-Lie groups and play an essential role in the study of quantum enveloping algebras. As a non-associative generalization of Lie algebras, Gerstenhaber~\cite{Ger63} and Vinberg~\cite{Vin63} introduced pre-Lie algebras in the study of deformation theory and convex homogeneous cones, respectively. Aguiar~\cite{Agu00} developed the coalgebra version while Bai~\cite{Bai08} systematically introduced the notion of pre-Lie (left-symmetric) bialgebras. For the sake of completeness and future reference, we briefly recall the definitions of pre-Lie algebras, pre-Lie coalgebras, and pre-Lie bialgebras in what follows.
\begin{defn}
A (left) \textbf{pre-Lie algebra} is a pair $(A,\circ)$ consisting of a vector space $A$ and a bilinear map $\circ : A \otimes A \to A$ (written $x \circ y$) such that for all $x,y,z \in A$,
\begin{equation}
(x \circ y)\circ z - x \circ (y \circ z)
=
(y \circ x)\circ z - y \circ (x \circ z). \mlabel{eq:preLie}
\end{equation}
\end{defn}

\begin{defn}
A \textbf{pre-Lie coalgebra} is a pair $(A,\Delta)$ consisting of a vector space $A$ and a linear map $\Delta : A \to A \otimes A$ such that for all $x \in A$,
\begin{equation}
(\Delta \otimes \mathrm{id}_A) \Delta(x) - (\mathrm{id}_A \otimes \Delta) \Delta(x)
=
(\tau \otimes \mathrm{id}_A)(\Delta \otimes \mathrm{id}_A) \Delta(x) - (\tau \otimes \mathrm{id}_A)(\mathrm{id}_A \otimes \Delta) \Delta(x), \mlabel{eq:preLieco}
\end{equation}
where $\tau(x \otimes y) = y \otimes x$ for all $x, y \in A$.
\end{defn}

\begin{defn}
A \textbf{pre-Lie bialgebra} is a triple $(A,\circ,\Delta)$ consisting of a pre-Lie algebra $(A,\circ)$ and a pre-Lie coalgebra $(A,\Delta)$
such that for all $x,y \in A$, the following conditions hold:
\begin{align*}
\Delta(x\circ y - y\circ x)
&= \bigl(L_x \otimes \mathrm{id}
   + \mathrm{id}\otimes (L_x - R_x)\bigr)\Delta(y)
 - \bigl(L_y \otimes \mathrm{id}
   + \mathrm{id}\otimes (L_y - R_y)\bigr)\Delta(x),\\
 \Delta(x\circ y) - \tau\bigl(\Delta(x\circ y)\bigr)
&= (\mathrm{id}\otimes R_y)\Delta(x)
   - \tau\bigl((\mathrm{id}\otimes R_y)\Delta(x)\bigr)
   + (L_x\otimes \mathrm{id} + \mathrm{id} \otimes L_x)\Delta(y)\\
   & \qquad - \tau\bigl((L_x\otimes \mathrm{id}+ \mathrm{id} \otimes L_x)\Delta(y)\bigr),
\end{align*}
where we define the left and right multiplication maps
$L,R : A \to \mathrm{End}(A)$ by
\[
L_x y = x \circ y
\quad \text{and} \quad
R_x y = y \circ x,
\qquad \forall x,y \in A.
\]
\end{defn}

Sheng, Tang and Zhu~\cite{ShengTangZhu21} investigated averaging operators on Lie algebras in a physical context where averaging operators are also referred to as embedding tensors, and established the corresponding cohomology theory by using derived brackets as the main tool. Hou and Cui~\cite{HC24} introduced averaging antisymmetric infinitesimal bialgebras by extending antisymmetric infinitesimal bialgebras to averaging algebras, characterized these structures via averaging Frobenius algebras and matched pairs, and gave their constructions from the Yang-Baxter equation and averaging dendriform algebras. Hou, Sheng and Zhou~\cite{HSZh26} studied averaging Lie bialgebras and showed the relations among their bialgebra structures, quadratic forms, matched pairs, and Manin triples. They also constructed the corresponding bialgebra structures by using the classical Yang-Baxter equations and Rota-Baxter operators.

\vspace{0.5cm}
{\noindent\bf The classical Yang-Baxter equation in pre-Lie algebra and Rota-Baxter operator.}
The Yang-Baxter equation was introduced by Yang~\cite{Yang67} in 1967 and  Baxter~\cite{Bax78} in 1971. Sklyanin, Takhtajan and Faddeev~\cite{FTS79} found solutions of the Yang-Baxter equation are closely related to the quantum inverse scattering method. Drinfeld~\cite{Dr83} showed that solutions of the classical Yang-Baxter equation give rise to certain Lie bialgebra structures.  Bai~\cite{Bai08} proved that solutions of the classical Yang-Baxter equation in the sense of pre-Lie algebras induce pre-Lie bialgebra structures, while Hou, Sheng and Zhou~\cite{HSZh26} showed that solutions of the classical Yang-Baxter equation in the framework of averaging Lie algebras give rise to averaging Lie bialgebra structures. Bai, Guo and Ni~\cite{BaiGN10} showed that Rota-Baxter operators on Lie algebras are also closely related to solutions of the classical Yang-Baxter equation. Kupershmidt~\cite{Kup99} introduced the notion of relative Rota-Baxter operators on Lie algebras, and a series of subsequent studies have shown that relative Rota-Baxter operators are intimately connected with solutions of the classical Yang-Baxter equation in various algebraic settings.

\vspace{0.5cm}
{\noindent\bf Main results and outline of this paper.}
As is well known, there is a relationship among dendriform algebras, associative algebras, Lie algebras, and pre-Lie algebras. Based on this, we focus primarily on the study of averaging pre-Lie bialgebra. Then, we summarize the main results in this paper by the following diagram:
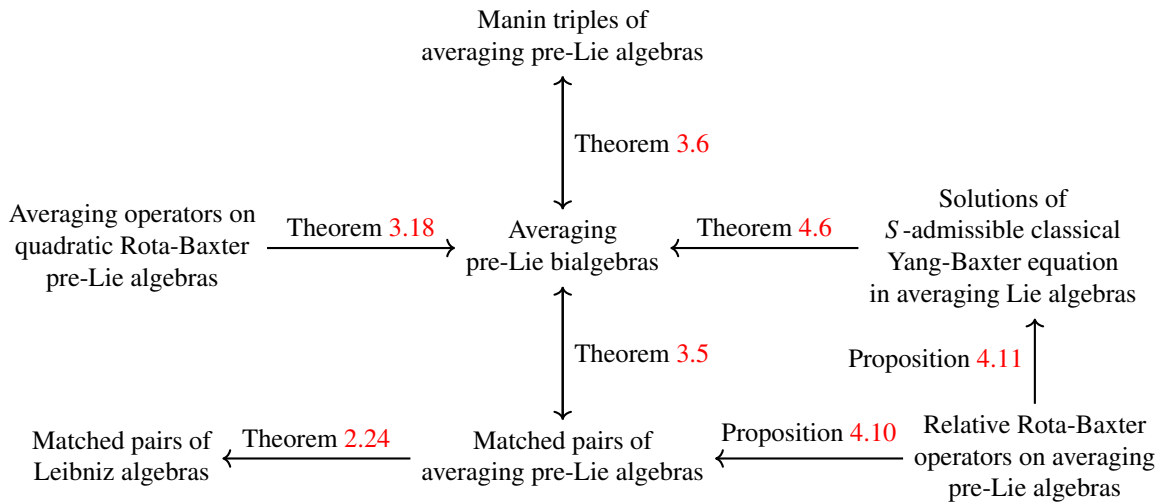
\begin{figure}[ht]
\centering
\begin{tikzpicture}[
  font=\scriptsize,
  node distance=2.2cm,
  every node/.style={align=center},
  arrow/.style={->, thick},
  transform shape,
  label/.style={font=\tiny}
]

\node (mtd) {Manin triples of\\ averaging pre-Lie algebras};
\node (aplb) [below of=mtd] {Averaging\\ pre-Lie bialgebras};
\node (mp) [below of=aplb] {Matched pairs of\\ averaging pre-Lie algebras};

\node (avqu) [left=2.2cm of aplb] {Averaging operators on\\ quadratic Rota-Baxter\\ pre-Lie algebras};
\node (leib) [left=2.2cm of mp] {Matched pairs of\\ Leibniz algebras};

\node (avlb) [right=2.2cm of mtd] {Averaging\\ Lie bialgebras};
\node (cybe) [right=2.2cm of aplb] {Solutions of\\ $S$-admissible classical\\ Yang-Baxter equation\\ in averaging Lie algebras};
\node (rb) [right=2.2cm of mp] {Relative Rota-Baxter\\ operators on averaging\\ pre-Lie algebras};

\draw[arrow] (mtd) -- node[label, left=1pt] {Theorem~\ref{theorem:mtriandaplbi}} (aplb);
\draw[arrow] (aplb) -- (mtd);

\draw[arrow] (avqu) -- node[label, above=1pt] {Theorem~\ref{theorem:avequtobi}} (aplb);
\draw[arrow] (aplb) -- node[label, left=1pt] {Theorem~\ref{theorem:mpoplbi}} (mp);
\draw[arrow] (mp) -- (aplb);

\draw[arrow] (mp) -- node[label, above=1pt] {Theorem~\ref{theorem:mpoplinducempolb}} (leib);
\draw[arrow] (rb) -- node[label, above=1pt] {Proposition~\ref{prop:mmoaverprelie}} (mp);
\draw[arrow] (cybe) -- node[label, above=1pt] {Theorem~\ref{theorem:sapplbi}} (aplb);
\draw[arrow] (aplb) -- node[label, sloped, above=1pt] {Theorem~\ref{theorem:aplb-avlb}} (avlb);
\draw[arrow] (rb.north) -- node[label, left=1pt] {Proposition~\ref{prop:rRBclYB}} (rb.north |- cybe.south);

\end{tikzpicture}
\end{figure}

The paper is organized as follows. In Section \ref{sec:Averaging pre-Lie algebras}, we study the representations of averaging pre-Lie algebras and show that representations of averaging pre-Lie algebras give rise to representations of Leibniz algebras. In Section \ref{sec:averbiandaverop}, we introduce the notions of matched pairs of averaging pre-Lie algebras, Manin triples of averaging pre-Lie algebras and averaging pre-Lie bialgebras and establish the equivalence of the three under suitable conditions. Furthermore, we introduce averaging operators on quadratic Rota-Baxter pre-Lie algebras and demonstrate that they induce an averaging pre-Lie bialgebra. In Section \ref{sec:adalybinapl}, we introduce the notion of the admissable classical Yang-Baxter equation in an averaging pre-Lie algebra whose solution gives rise to an averaging pre-Lie bialgebra. Then we introduce the notion of relative Rota-Baxter operators on an averaging Lie algebra with respect to a representation, which can give rise to solutions of the classical Yang-Baxter equation in the semidirect product averaging pre-Lie algebra. In Section \ref{sec:balancedaverpl}, we introduce a new notion of averaging Lie bialgebras and show that, under suitable conditions, an averaging Lie bialgebra can be obtained from an averaging pre-Lie bialgebra. Notably, when the averaging operators associated with the pre-Lie algebra and pre-Lie coalgebra coincide, the definition of an averaging Lie bialgebra we introduce (Definition~\ref{def:averagingLiebialgebra}) aligns with the notion of averaging Lie bialgebras presented in ~\cite[Definition~4.4]{HSZh26}.

\smallskip
\textbf{Notation.} Throughout this paper, we fix a field $\bfk$ of characteristic $0$ and all the vector spaces and algebras are finite-dimensional. Linear maps and tensor products are taken over $\bfk.$ Unless stated otherwise, for any vector space $V$, $V^*$ denotes the dual space of $V$, and this notation is also adopted for linear maps. The angle bracket $\langle\cdot, \cdot\rangle$ stands for the canonical pairing between a space and its dual.

\section{Averaging pre-Lie algebras}\label{sec:Averaging pre-Lie algebras}

In this section, we introduce representations of averaging pre-Lie algebras and the concepts of matched pairs and Manin triples of averaging pre-Lie algebras. Moreover, we show that a quadratic averaging pre-Lie algebra gives rise to an isomorphism from the regular representation to the coregular representation.

\subsection{Averaging pre-Lie algebras and its representations}

In this subsection, we recall the definition of averaging pre-Lie algebras and their representations, and show that the dual representations exists.

\begin{defn}\cite{HSZh26}
Let $(A,\circ)$ be a pre-Lie algebra and $P : A \to A$ a linear map.
If
\begin{equation}
P(x) \circ P(y) = P\Big( P(x) \circ y \Big) = P\Big(x \circ P(y) \Big) ,
\qquad \forall x,y \in A, \mlabel{eq:averop}
\end{equation}
then we call $P$ an \textbf{averaging operator} on $(A,\circ)$ and the pair $\big( (A,\circ), P \big)$ an \textbf{averaging pre-Lie algebra}.
\end{defn}

\begin{defn}
Let $\big( (A,\circ), P \big)$ and $\big( (A',\circ '), P' \big)$ be two averaging pre-Lie algebras.
A linear map $f : A \to A'$ is called a \textbf{homomorphism from
$\big( (A,\circ), P \big)$ to $\big( (A',\circ'), P' \big)$} if
\[
f(x\circ y)=f(x)\circ' f(y),
\qquad \forall\,x,y\in A,
\]
and
\[
f(P(v)) = P'(f(v)),
\quad \forall\,  v \in A.
\]
\end{defn}

\begin{exam}
Let $(A,\circ)$ be a pre-Lie algebra.
\begin{enumerate}
\item $P = \mathrm{id}_A$ is an averaging operator on $(A,\circ)$;

\item If a linear map $P : A \to A$ commutes with the left and right multiplication,
i.e., for all $x \in A$,
\[
P L_x = L_x P \quad \text{and } \quad P R_x = R_x P,
\]
then $P$ is an averaging operator.
\end{enumerate}
\end{exam}

\begin{exam}
Let $\bfk$ be a field and let
\[
A=UT_2(\mathbf{k})
=\left\{\begin{pmatrix}\mathrm{a}&\mathrm{b}\\0&\mathrm{c}\end{pmatrix}\ \middle|\ \mathrm{a}, \mathrm{b}, \mathrm{c} \in \mathbf{k} \right\}
\]
be the associative algebra of $2\times 2$ upper triangular matrices over $\bfk$, with multiplication given by matrix multiplication.
Define a linear map
\[
\begin{array}{rcl}
R:\qquad A &\longrightarrow& A\\[2mm]
\begin{pmatrix}a&b\\0&c\end{pmatrix}
&\longmapsto&
\begin{pmatrix}a&0\\0&c\end{pmatrix}
\end{array}
\]
and a bilinear operation $\circ$ on $A$ by
\[
x\circ y:=R(x)y-yR(x)-xy,
\qquad \forall\,x,y\in A.
\]
Then by \cite[Lemma~8.6]{WBLS}, $(A,\circ)$ is a pre-Lie algebra. Obviously, $R$ is an averaging operator on $(A,\circ)$.
Thus the triple $(A,\circ,R)$ is an averaging pre-Lie algebra.
\end{exam}

The following proposition shows that an averaging pre-Lie algebra can induce a Leibniz algebra.
\begin{prop}\label{prop:ilbniz}
Let $\bigl((A,\circ),P\bigr)$ be an averaging pre-Lie algebra. Define a bilinear operation
\[
[x,y]_P \;:=\; P(x)\circ y \;-\; y\circ P(x),\qquad \forall\,x,y\in A .
\]
Then $(A,[\cdot ,\cdot ]_P)$ is a Leibniz algebra, i.e.
\[
[[x,y]_P,z]_P \;+\; [y,[x,z]_P]_P \;=\; [x,[y,z]_P]_P,\qquad \forall\,x,y,z\in A.
\]
The Leibniz algebra $(A,[\cdot ,\cdot ]_P)$ is called \textbf{the induced Leibniz algebra} of the averaging pre-Lie algebra, denoted by $A_P$.
\end{prop}
\begin{proof}
For all $x,y,z\in A$, we have
\begin{align*}
&\quad [[x,y]_P,z]_P + [y,[x,z]_P]_P\\
&= P([x,y]_P)\circ z - z\circ P([x,y]_P) + P(y)\circ[x,z]_P - [x,z]_P\circ P(y)\\
&= P\bigl(P(x)\circ y - y\circ P(x)\bigr)\circ z
   - z\circ P\bigl(P(x)\circ y - y\circ P(x)\bigr)\\
&\qquad + P(y)\circ\bigl(P(x)\circ z - z\circ P(x)\bigr)
   - \bigl(P(x)\circ z - z\circ P(x)\bigr)\circ P(y)\\
&= P(P(x)\circ y)\circ z - P(y\circ P(x))\circ z
   - z\circ P(P(x)\circ y) + z\circ P(y\circ P(x))\\
&\qquad + P(y)\circ(P(x)\circ z) - P(y)\circ(z\circ P(x))
   - (P(x)\circ z)\circ P(y) + (z\circ P(x))\circ P(y)\\
&\overset{\eqref{eq:averop}}{=} (P(x)\circ P(y))\circ z - (P(y)\circ P(x))\circ z
   - z\circ(P(x)\circ P(y)) + z\circ(P(y)\circ P(x))\\
&\qquad + P(y)\circ(P(x)\circ z) - P(y)\circ(z\circ P(x))
   - (P(x)\circ z)\circ P(y) + (z\circ P(x))\circ P(y)\\
&\overset{\eqref{eq:preLie}}{=} P(x)\circ(P(y)\circ z) + (P(y)\circ P(x))\circ z
   - P(y)\circ(P(x)\circ z) - (P(y)\circ P(x))\circ z\\
&\qquad - z\circ(P(x)\circ P(y)) + z\circ(P(y)\circ P(x))
   + P(y)\circ(P(x)\circ z) - P(y)\circ(z\circ P(x))\\
&\qquad - (P(x)\circ z)\circ P(y) + (z\circ P(x))\circ P(y)\\
&= P(x)\circ(P(y)\circ z)
   - P(y)\circ(z\circ P(x))
   - (P(x)\circ z)\circ P(y)
   + (z\circ P(x))\circ P(y)\\
&\qquad - z\circ(P(x)\circ P(y))
   + z\circ(P(y)\circ P(x))\\
&= P(x)\circ(P(y)\circ z - z\circ P(y))
   - (P(y)\circ z - z\circ P(y))\circ P(x)\\
&= P(x)\circ[y,z]_P - [y,z]_P\circ P(x)\\
&= [x,[y,z]_P]_P.
\end{align*}
This completes the proof.
\end{proof}

\begin{defn}\cite{Bai08}
A \textbf{representation of a pre-Lie algebra $(A,\circ)$} is a triple $(V,\rho,\varphi)$, where
\begin{enumerate}
\item $V$ is a vector space;

\item $\rho,\varphi : A \to \mathrm{End}(V)$ are linear maps;

\item for all $x, y \in A$ and $v \in V$, we have
\begin{align}
&\rho(x \circ y - y \circ x)v
=
\rho(x)(\rho(y)v) - \rho(y)(\rho(x)v), \mlabel{eq:ropl1}\\
&\varphi(x \circ y)v
=
\rho(x)(\varphi(y)v) - \varphi(y)(\rho(x)v) + \varphi(y)(\varphi(x)v).\mlabel{eq:ropl2}
\end{align}
\end{enumerate}
\end{defn}

Based on the above definition, we propose the notion of representations of averaging pre-Lie algebras.
\begin{defn}
Let $\big( (A,\circ), P \big)$ be an averaging pre-Lie algebra, a \textbf{representation of $\big( (A,\circ), P \big)$} is a pair
$\big( (V,\rho,\varphi), \alpha \big)$, where
\begin{enumerate}
\item $(V,\rho,\varphi)$ is a representation of $(A,\circ)$;

\item $\alpha : V \to V$ is a linear map;

\item for all $x \in A$ and $v \in V$, we have the following compatibility conditions:
\begin{align}
\rho(P(x))\alpha(v) &= \alpha\big(\rho(P(x))v\big) = \alpha\big(\rho(x)\alpha(v)\big),  \mlabel{eq:rapL1}\\
\varphi(P(x))\alpha(v) &= \alpha\big(\varphi(P(x))v\big) = \alpha\big(\varphi(x)\alpha(v)\big). \mlabel{eq:rapL2}
\end{align}
\end{enumerate}
\end{defn}

\begin{defn}
Let $\big( (V_1,\rho_1,\varphi_1), \alpha_1\big)$ and $\big( (V_2,\rho_2,\varphi_2), \alpha_2\big)$ be two representations of an averaging pre-Lie algebra $\big( (A,\circ), P \big)$.
A linear map $f : V_1 \to V_2$ is called a \textbf{homomorphism from
$\big( (V_1,\rho_1,\varphi_1), \alpha_1\big)$ to $\big( (V_2,\rho_2,\varphi_2), \alpha_2\big)$} if for any $x \in A,\ v \in V_1$, we have
\begin{equation}
\rho_2(x) f(v) = f(\rho_1(x)v),
\qquad
\varphi_2(x) f(v) = f(\varphi_1(x)v),
\qquad
f(\alpha_1(v)) = \alpha_2(f(v)).\mlabel{eq:homocondi}
\end{equation}
Furthermore, if $f$ is invertible, then $f$ is called an \textbf{isomorphism}.
\end{defn}

\begin{exam}
Let $\big((A,\circ), P \big)$ be an averaging pre-Lie algebra. Then $\big((A,L,R),P\big)$ is a representation of $\big((A,\circ),P\big)$, called the \textbf{regular representation} of $\big((A,\circ),P\big)$.
\end{exam}

The following proposition characterizes semi-direct product averaging pre-Lie algebras via representations.

\begin{prop} \label{prop:semi-repre}
Let $\big((A,\circ),P\big)$ be an averaging pre-Lie algebra, $(V,\rho,\varphi)$ be a representation of $(A,\circ)$ and $\alpha : V \to V$ a linear map.
Define a product $\circ_{\heartsuit}$ on $A \oplus V$ by
\begin{equation}
(x+u)\circ_{\heartsuit}(y+v)
:=
x \circ y + \rho(x)v + \varphi(y)u,
\qquad \forall x,y \in A,\ u,v \in V.  \mlabel{eq:sdp}
\end{equation}
and a linear map $P_{A\oplus V} : A \oplus V \to A \oplus V$ by
\begin{equation}
P_{A\oplus V}(x+v)
:=
(P+\alpha)(x+v)
=
P(x) + \alpha(v). \mlabel{eq:sda}
\end{equation}
Then $A \oplus V$ equipped with the product Eq.~\eqref{eq:sdp} and the linear map
Eq.~\eqref{eq:sda} is an averaging pre-Lie algebra, denoted by $\big(A \ltimes_{\rho,\varphi} V, P + \alpha\big)$, if and only if
$\big((V,\rho,\varphi),\alpha\big)$ is a representation of $\big((A,\circ),P\big)$.
In this case, we call $(A \ltimes_{\rho,\varphi} V, P + \alpha)$ a \textbf{semi-direct product of averaging pre-Lie algebra}.
\begin{proof}
First, by \cite[Proposition~3.1]{Bai08}, we know that $\big(A \oplus V, \circ_{\heartsuit}\big) = :A \ltimes_{\rho,\varphi} V$ is a pre-Lie algebra if and only if $(V, \rho, \varphi)$ is a representation of $(A,\circ)$.

Second, by a straightforward computation, we have
\begin{align*}
P_{A\oplus V}(x+u) \circ_{\heartsuit} P_{A\oplus V}(y+v)
&\overset{\eqref{eq:sda}}{=} (P(x) + \alpha(u)) \circ_{\heartsuit} (P(y) + \alpha(v))\\
&\overset{\eqref{eq:sdp}}{=} P(x) \circ P(y) + \rho(P(x))\alpha(v) + \varphi(P(y))\alpha(u),\\
P_{A\oplus V}\big(P_{A\oplus V}(x+u) \circ_{\heartsuit} (y+v) \big)
&\overset{\eqref{eq:sda}}{=} P_{A\oplus V}\big((P(x) + \alpha(u)) \circ_{\heartsuit} (y+v) \big)\\
&\overset{\eqref{eq:sdp}}{=} P_{A\oplus V}\big(P(x) \circ y + \rho(P(x))v + \varphi(y)\alpha(u)\big)\\
&\overset{\eqref{eq:sda}}{=} P\big(P(x) \circ y\big) + \alpha\big(\rho(P(x))v\big) + \alpha\big(\varphi(y)\alpha(u)\big),
\end{align*}
we know that $P_{A\oplus V}$ is an averaging operator on $A \ltimes_{\rho,\varphi} V$ if and only if Eq.~\eqref{eq:rapL1} and Eq.~\eqref{eq:rapL2} hold.
Thus, the proof is completed.
\end{proof}
\end{prop}

Next, we consider the dual representation of an averaging pre-Lie algebra.

\begin{prop}\label{prop:pplre}
Let $\big((A,\circ),P\big)$ be an averaging pre-Lie algebra, $(V,\rho,\varphi)$ a representation of $(A,\circ)$, and $\beta:V\to V$ a linear map. Define linear maps
$\rho^*,\varphi^*:A\to \mathrm{End}(V^*)$ by
\begin{equation}
\langle \rho^*(x)v^*,\, v\rangle=-\langle v^*,\, \rho(x)v\rangle, \qquad
\langle \varphi^*(x)v^*,\, v\rangle=-\langle v^*,\, \varphi(x)v\rangle,   \mlabel{eq:dual}
\end{equation}
for all $x\in A$, $v^*\in V^*$ and $v\in V$.
Then $\big( (V^*,\rho^*-\varphi^*,-\varphi^*), \beta^* \big)$ is a representation of $\big((A,\circ),P\big)$
if and only if for all $x\in A$ and $v\in V$, satisfying
\begin{equation}
\rho(P(x))\beta(v) = \beta\big(\rho(P(x))v\big) = \beta\big(\rho(x)\beta(v)\big), \mlabel{eq:beta1}
\end{equation}
\begin{equation}
\varphi(P(x))\beta(v) = \beta\big(\varphi(P(x))v\big) = \beta\big(\varphi(x)\beta(v)\big). \mlabel{eq:beta2}
\end{equation}
\end{prop}
\begin{proof}
First, by \cite[Proposition~3.3]{Bai08}, we know that $(V^*,\rho^*-\varphi^*,-\varphi^*)$ is a representation of $(A, \circ)$.

Second, by a straightforward computation, for $\forall x \in A, v \in V$ and $v^* \in V^*$,
\begin{align*}
\langle -\varphi^*(P(x))\beta^*(v^*),\, v\rangle
\overset{\eqref{eq:dual}}{=} \langle \beta^*(v^*),\, \varphi(P(x))v\rangle
= \langle v^*,\, \beta\big(\varphi(P(x))v\big)\rangle,
\end{align*}
\begin{align*}
\langle \beta^*\bigl(-\varphi^*(x)\beta^*(v^*)\bigr),\, v\rangle
= \langle -\varphi^*(x)\beta^*(v^*),\, \beta(v)\rangle
\overset{\eqref{eq:dual}}{=} \langle \beta^*(v^*),\, \varphi(x)\beta(v)\rangle
= \langle v^*,\, \beta\bigl(\varphi(x)\beta(v)\bigr)\rangle,
\end{align*}
\begin{align*}
\langle \beta^*\bigl(-\varphi^*(P(x))v^*\bigr),\, v\rangle
= \langle -\varphi^*(P(x))v^*,\, \beta(v)\rangle
\overset{\eqref{eq:dual}}{=}\langle v^*,\, \varphi(P(x))\beta(v)\rangle,
\end{align*}
we know that Eq.~\eqref{eq:rapL2} holds for $-\varphi^*$ if and only if Eq.~\eqref{eq:beta2} holds. Similarly, Eq.~\eqref{eq:rapL1} holds for $\rho^*-\varphi^*$ if and only if for all $x\in A$ and $v \in V$, satisfying
\begin{align*}
\beta\bigl(\rho(P(x))v\bigr) - \beta\bigl(\varphi(P(x))v\bigr)
= \rho(P(x))\beta(v) - \varphi(P(x))\beta(v)
= \beta\bigl(\rho(x)\beta(v)\bigr) - \beta\bigl(\varphi(x)\beta(v)\bigr),
\end{align*}
which is equivalent to Eq.~\eqref{eq:beta1} by Eq.~\eqref{eq:beta2}.
This completes the proof.
\end{proof}

Now, we apply the above result to the regular representation.

\begin{coro}\label{coro:admiss}
Let $\big((A,\circ),P\big)$ be an averaging pre-Lie algebra. For a linear map $S:A\to A$, the pair $\big((A^*,L^*-R^*,-R^*),S^*\big)$ is a representation of $\big((A,\circ),P\big)$ if and only if for all $x,y\in A$, satisfying
\begin{equation}
P(x)\circ S(y) = S\big(P(x)\circ y\big) = S\big(x\circ S(y)\big),  \mlabel{eq:regu1}
\end{equation}
\begin{equation}
S(x)\circ P(y) = S\big(x\circ P(y)\big) = S\big(S(x)\circ y\big).  \mlabel{eq:regu2}
\end{equation}
\end{coro}

\begin{coro}
Let $\big((A,\circ), P \big)$ be an averaging pre-Lie algebra. Then $\big((A^*,L^*-R^*,-R^*),P^*\big)$ is a representation of $\big((A,\circ),P\big)$, called the \textbf{coregular representation} of $\big((A,\circ),P\big)$.
\end{coro}
\begin{proof}
This follows directly by taking $S = P$ in Corollary ~\ref{coro:admiss}.
\end{proof}

\begin{prop}
Let $\big((A,\circ),P\big)$ be an averaging pre-Lie algebra and $\big( (V,\rho, \varphi), \beta \big)$ a representation of $\big((A,\circ),P\big)$. Then the following conditions are equivalent:
\begin{enumerate}
\item $\big( (V^*,\rho^*, \varphi^*), \beta^* \big)$ is a representation of $\big((A,\circ),P\big)$;
\item $\big( (V,\rho-\varphi,-\varphi), \beta \big)$ is a representation of $\big((A,\circ),P\big)$;
\item $\varphi(x)\varphi(y)= - \varphi(y)\varphi(x)$ for any $x, y \in A$.
\end{enumerate}
\end{prop}
\begin{proof}
First, by \cite[Proposition~3.4]{Bai08}, we know that the following conditions are equivalent:
\begin{enumerate}
\item $ (V^*,\rho^*, \varphi^*)$ is a representation of $(A,\circ)$;
\item $(V,\rho-\varphi,-\varphi)$ is a representation of $(A,\circ)$;
\item $\varphi(x)\varphi(y)= - \varphi(y)\varphi(x)$ for any $x, y \in A$.
\end{enumerate}

Second, by some simple calculations, we know that Eq.~\eqref{eq:rapL1} holds for $\rho^*$ and $\beta^*$ if and only if Eq.~\eqref{eq:rapL1} holds for $\rho$ and $\beta$; Eq.~\eqref{eq:rapL2} holds for $\varphi^*$ and $\beta^*$ if and only if Eq.~\eqref{eq:rapL2} holds for $\varphi$ and $\beta$.

Finally, by Proposition~\ref{prop:pplre}, we obtain $(V,\rho-\varphi,-\varphi)$ is a representation of $(A,\circ)$.
\end{proof}

\begin{defn}
Let $((A,\circ),P)$ be an averaging pre-Lie algebra, $(V,\rho,\varphi)$ a representation of $(A,\circ)$, and $\beta:V\to V$ a linear map. The averaging pre-Lie algebra $\big((A,\circ),P\big)$ is called
\textbf{$\beta$-admissible with respect to $(V,\rho,\varphi)$} if Eq.~\eqref{eq:beta1} and Eq.~\eqref{eq:beta2} hold.
In particular, when Eq.~\eqref{eq:regu1} and Eq.~\eqref{eq:regu2} hold, we say that $\big((A,\circ),P\big)$ is \textbf{$S$-admissible}.
\end{defn}

In the rest of this section, we show that a representation of an averaging pre-Lie algebra induces a representation of the underlying Leibniz algebra.

\begin{defn}\cite{HSZh26}
A \textbf{representation of a Leibniz algebra} $(\mathcal{G},[\cdot,\cdot]_{\mathcal{G}})$ is a triple $(W;\rho^L,\rho^R)$, where
\begin{enumerate}
\item $W$ is a vector space;
\item $\rho^L,\rho^R:\mathcal{G} \to \mathfrak{gl}(W)$ are linear maps;
\item the following equalities hold for all $x,y\in\mathcal{G}$,
\begin{align}
\rho^L\bigl([x,y]_{\mathcal{G}}\bigr) &= \rho^L(x)\rho^L(y)-\rho^L(y)\rho^L(x), \mlabel{eq:rola1}\\
\rho^R\bigl([x,y]_{\mathcal{G}}\bigr) &= \rho^L(x)\rho^R(y)-\rho^R(y)\rho^L(x), \mlabel{eq:rola2}\\
\rho^R(y)\circ\rho^L(x) &= -\,\rho^R(y)\circ\rho^R(x). \mlabel{eq:rola3}
\end{align}
\end{enumerate}
\end{defn}

\begin{prop}
Let $\bigl((V,\rho,\varphi),\alpha\bigr)$ be a representation of an averaging pre-Lie algebra $\bigl((A,\circ),P\bigr)$. Define linear maps $\rho^L,\rho^R:A\to \mathfrak{gl}(V)$ by
\begin{align*}
\rho^L(x)\;:=\;\rho(P(x))-\varphi(P(x)),\qquad
\rho^R(x)(\xi)\;:=\;-\rho(x)\bigl(\alpha(\xi)\bigr)+\varphi(x)\bigl(\alpha(\xi)\bigr),
\ \ \forall\,x\in A,\ \xi\in V.
\end{align*}
Then $\bigl(V;\rho^L,\rho^R\bigr)$ is a representation of the induced Leibniz algebra $A_P$.
\end{prop}
\begin{proof}
For any $x, y\in A$, $\xi \in V$, we need to prove Eqs.~\eqref{eq:rola1}-~\eqref{eq:rola3}.

First, we prove Eq.~\eqref{eq:rola1}.
\begin{align*}
\rho^L([x,y]_P)(\xi)
&= \rho^L(P(x)\circ y - y\circ P(x))(\xi)\\
&= \rho\bigl(P(P(x)\circ y - y\circ P(x))\bigr)(\xi)
   - \varphi\bigl(P(P(x)\circ y - y\circ P(x))\bigr)(\xi)\\
&\overset{\eqref{eq:averop}}{=} \rho\bigl(P(x)\circ P(y) - P(y)\circ P(x)\bigr)(\xi)
   - \varphi\bigl(P(x)\circ P(y)\bigr)(\xi)
   + \varphi\bigl(P(y)\circ P(x)\bigr)(\xi)\\
&\overset{\eqref{eq:ropl1} \eqref{eq:ropl2}}{=} \rho(P(x))\bigl(\rho(P(y))(\xi)\bigr)
   - \rho(P(y))\bigl(\rho(P(x))(\xi)\bigr)\\
&\qquad
   - \rho(P(x))\bigl(\varphi(P(y))(\xi)\bigr)
   + \varphi(P(y))\bigl(\rho(P(x))(\xi)\bigr)
   - \varphi(P(y))\bigl(\varphi(P(x))(\xi)\bigr)\\
&\qquad
   + \rho(P(y))\bigl(\varphi(P(x))(\xi)\bigr)
   - \varphi(P(x))\bigl(\rho(P(y))(\xi)\bigr)
   + \varphi(P(x))\bigl(\varphi(P(y))(\xi)\bigr)\\
&= \rho(P(x))\bigl(\rho(P(y))(\xi) - \varphi(P(y))(\xi)\bigr)
   - \rho(P(y))\bigl(\rho(P(x))(\xi) - \varphi(P(x))(\xi)\bigr)\\
&\qquad
   - \varphi(P(x))\bigl(\rho(P(y))(\xi) - \varphi(P(y))(\xi)\bigr)
   + \varphi(P(y))\bigl(\rho(P(x))(\xi) - \varphi(P(x))(\xi)\bigr)\\
&= \bigl(\rho(P(x))-\varphi(P(x))\bigr)
   \bigl(\rho(P(y))-\varphi(P(y))\bigr)(\xi)\\
&\qquad
   - \bigl(\rho(P(y))-\varphi(P(y))\bigr)
     \bigl(\rho(P(x))-\varphi(P(x))\bigr)(\xi)\\
&= \bigl(\rho^L(x)\rho^L(y)-\rho^L(y)\rho^L(x)\bigr)(\xi).
\end{align*}
Next, we prove Eq.~\eqref{eq:rola2}.
\begin{align*}
\rho^R([x,y]_P)(\xi)
&= \rho^R(P(x)\circ y - y\circ P(x))(\xi)\\
&= -\,\rho(P(x)\circ y - y\circ P(x))\,\alpha(\xi)
   + \varphi(P(x)\circ y)\,\alpha(\xi)
   - \varphi(y\circ P(x))\,\alpha(\xi)\\
&\overset{\eqref{eq:ropl1} \eqref{eq:ropl2}}{=} -\,\rho(P(x))\bigl(\rho(y)\alpha(\xi)\bigr)
   + \rho(y)\bigl(\rho(P(x))\alpha(\xi)\bigr)\\
&\qquad
   + \rho(P(x))\bigl(\varphi(y)\alpha(\xi)\bigr)
   - \varphi(y)\bigl(\rho(P(x))\alpha(\xi)\bigr)
   + \varphi(y)\bigl(\varphi(P(x))\alpha(\xi)\bigr)\\
&\qquad
   - \rho(y)\bigl(\varphi(P(x))\alpha(\xi)\bigr)
   + \varphi(P(x))\bigl(\rho(y)\alpha(\xi)\bigr)
   - \varphi(P(x))\bigl(\varphi(y)\alpha(\xi)\bigr)\\
&\overset{\eqref{eq:rapL1} \eqref{eq:rapL2}}{=} -\,\rho(P(x))\bigl(\rho(y)\alpha(\xi)\bigr)
   + \varphi(P(x))\bigl(\rho(y)\alpha(\xi)\bigr)
   + \rho(P(x))\bigl(\varphi(y)\alpha(\xi)\bigr)\\
&\qquad
   - \varphi(P(x))\bigl(\varphi(y)\alpha(\xi)\bigr)
   + \rho(y)\alpha\bigl(\rho(P(x))(\xi)\bigr)
   - \varphi(y)\alpha\bigl(\rho(P(x))(\xi)\bigr)\\
&\qquad
   - \rho(y)\alpha\bigl(\varphi(P(x))(\xi)\bigr)
   + \varphi(y)\alpha\bigl(\varphi(P(x))(\xi)\bigr)\\
&= \rho^L(x)\bigl(-\rho(y)\alpha(\xi)+\varphi(y)\alpha(\xi)\bigr)
   - \rho^R(y)\bigl(\rho(P(x))(\xi)-\varphi(P(x))(\xi)\bigr)\\
&= \bigl(\rho^L(x)\rho^R(y)-\rho^R(y)\rho^L(x)\bigr)(\xi),
\end{align*}
Finally, we prove Eq.~\eqref{eq:rola3}.
\begin{align*}
\bigl(\rho^R(y)\rho^L(x)+\rho^R(y)\rho^R(x)\bigr)(\xi)
&= \rho^R(y)\bigl(\rho(P(x))\xi-\varphi(P(x))\xi\bigr)
   + \rho^R(y)\bigl(-\rho(x)\alpha(\xi)+\varphi(x)\alpha(\xi)\bigr)\\
&= -\,\rho(y)\alpha\bigl(\rho(P(x))\xi\bigr)
   + \varphi(y)\alpha\bigl(\rho(P(x))\xi\bigr)
   + \rho(y)\alpha\bigl(\varphi(P(x))\xi\bigr)\\
&\qquad
   - \varphi(y)\alpha\bigl(\varphi(P(x))\xi\bigr)
   + \rho(y)\alpha\bigl(\rho(x)\alpha(\xi)\bigr)
   - \varphi(y)\alpha\bigl(\rho(x)\alpha(\xi)\bigr)\\
&\qquad
   - \rho(y)\alpha\bigl(\varphi(x)\alpha(\xi)\bigr)
   + \varphi(y)\alpha\bigl(\varphi(x)\alpha(\xi)\bigr)\\
&= 0.
\end{align*}
This completes the proof.
\end{proof}

\subsection{Matched pair of averaging pre-Lie algebras}
In this subsection, we introduce the notion of matched pairs of averaging pre-Lie algebras and show that such a matched pair gives rise to an induced matched pairs of Leibniz algebra.

\begin{defn}\cite{Bai08}
A matched pair of the pre-Lie algebras is a sextuple $\big((A,\circ_A),(\mathfrak{b},\circ_{\mathfrak{b}}),\rho_A,\varphi_A,\rho_{\mathfrak{b}},\varphi_{\mathfrak{b}}\big)$, where
\begin{enumerate}
\item $(A,\circ_A)$ and $(\mathfrak{b},\circ_{\mathfrak{b}})$ are pre-Lie algebras;
\item $(\mathfrak{b},\rho_A,\varphi_A)$ is a representation of $(A,\circ_A)$;
\item $(A,\rho_{\mathfrak{b}},\varphi_{\mathfrak{b}})$ is a representation of $(\mathfrak{b},\circ_{\mathfrak{b}})$;
\item for all $x,y \in A$ and $a,b \in \mathfrak b$, the following equations hold:
\begin{align}
&\rho_A(x)\bigl(a \circ_{\mathfrak b} b\bigr)
= -\rho_A\bigl(\rho_{\mathfrak b}(a)x-\varphi_{\mathfrak b}(a)x\bigr)b
 +\bigl(\rho_A(x)a-\varphi_A(x)a\bigr)\circ_{\mathfrak b} b
 +\varphi_A\bigl(\varphi_{\mathfrak b}(b)x\bigr)a
 +a\circ_{\mathfrak b}\bigl(\rho_A(x)b\bigr), \mlabel{eq:mppl1} \\
&\varphi_A(x)\bigl(a \circ_{\mathfrak b} b-b \circ_{\mathfrak b} a\bigr)
= \varphi_A\bigl(\rho_{\mathfrak b}(b)x\bigr)a
 -\varphi_A\bigl(\rho_{\mathfrak b}(a)x\bigr)b
 +a\circ_{\mathfrak b}\bigl(\varphi_A(x)b\bigr)
 -b\circ_{\mathfrak b}\bigl(\varphi_A(x)a\bigr), \mlabel{eq:mppl2}\\
&\rho_{\mathfrak b}(a)\bigl(x \circ_A y\bigr)
= -\rho_{\mathfrak b}\bigl(\rho_A(x)a-\varphi_A(x)a\bigr)y
 +\bigl(\rho_{\mathfrak b}(a)x-\varphi_{\mathfrak b}(a)x\bigr)\circ_A y
 +\varphi_{\mathfrak b}\bigl(\varphi_A(y)a\bigr)x
 +x\circ_A\bigl(\rho_{\mathfrak b}(a)y\bigr), \mlabel{eq:mppl3}\\
&\varphi_{\mathfrak b}(a)\bigl(x \circ_A y-y \circ_A x\bigr)
= \varphi_{\mathfrak b}\bigl(\rho_A(y)a\bigr)x
 -\varphi_{\mathfrak b}\bigl(\rho_A(x)a\bigr)y
 +x\circ_A\bigl(\varphi_{\mathfrak b}(a)y\bigr)
 -y\circ_A\bigl(\varphi_{\mathfrak b}(a)x\bigr). \mlabel{eq:mppl4}
\end{align}
\end{enumerate}
\end{defn}

Now, we give the definition of a matched pair of averaging pre-Lie algebras.

\begin{defn}\label{def:mpoaverprelie}
A \textbf{matched pair of averaging pre-Lie algebras} is a sextuple
\[\Bigl(\bigl((A,\circ_A),P_A\bigr),\ \bigl((\mathfrak{b},\circ_{\mathfrak{b}}),P_{\mathfrak{b}}\bigr),\
\rho_A,\varphi_A,\rho_{\mathfrak{b}},\varphi_{\mathfrak{b}}\Bigr),\]
where
\begin{enumerate}
\item $\big((A,\circ_A),P_A\big)$ and $\big((\mathfrak{b},\circ_{\mathfrak{b}}),P_{\mathfrak{b}}\big)$ are averaging pre-Lie algebras;
\item $(\big(\mathfrak{b},\rho_A,\varphi_A),P_{\mathfrak{b}}\big)$ is a representation of $\big((A,\circ_A),P_A\big)$;
\item $\big((A,\rho_{\mathfrak{b}},\varphi_{\mathfrak{b}}),P_A\big)$ is a representation of $\big((\mathfrak{b},\circ_{\mathfrak{b}}),P_{\mathfrak{b}}\big)$;
\item $\big((A,\circ_A),(\mathfrak{b},\circ_{\mathfrak{b}}),\rho_A,\varphi_A,\rho_{\mathfrak{b}},\varphi_{\mathfrak{b}}\big)$
      is a matched pair of the pre-Lie algebras $(A,\circ_A)$ and $(\mathfrak{b},\circ_{\mathfrak{b}})$.
\end{enumerate}
\end{defn}

Given a matched pair of averaging pre-Lie algebras, their double admits an averaging pre-Lie algebra structure.

\begin{prop}\label{prop:mpopl}
Let $\bigl((A,\circ_A),P_A\bigr)$ and $\bigl((\mathfrak b,\circ_{\mathfrak b}),P_{\mathfrak b}\bigr)$ be two averaging pre-Lie algebras.
Assume that $\bigl((\mathfrak b,\rho_A,\varphi_A),P_{\mathfrak b}\bigr)$ is a representation of $\bigl((A,\circ_A),P_A\bigr)$ and $\bigl((A,\rho_{\mathfrak b},\varphi_{\mathfrak b}),P_A\bigr)$ is a representation of $\bigl((\mathfrak b,\circ_{\mathfrak b}),P_{\mathfrak b}\bigr)$.
Define a product $\circ_\star$ on $A\oplus \mathfrak b$ by
\begin{equation}
(x+a)\circ_\star (y+b)
= x\circ_A y + \rho_{\mathfrak b}(a)y + \varphi_{\mathfrak b}(b)x
+ a\circ_{\mathfrak b} b + \rho_A(x)b + \varphi_A(y)a, \mlabel{eq:zh1}
\end{equation}
and a linear map $P_\star$ on $A\oplus \mathfrak b$ by
\begin{equation}
P_\star(x+a):=P_A(x)+P_{\mathfrak b}(a), \qquad x,y\in A,\ a,b\in\mathfrak b. \mlabel{eq:zh2}
\end{equation}
Then $\bigl((A\oplus \mathfrak b,\circ_\star),P_\star\bigr)$ is an averaging pre-Lie algebra if and only if Eqs.~\eqref{eq:mppl1}-~\eqref{eq:mppl4} hold.
In this case, we denote $\bigl((A\oplus \mathfrak b,\circ_\star),P_\star\bigr)$ by $(A\bowtie \mathfrak b,\; P_A\oplus P_{\mathfrak b})$.
\end{prop}
\begin{proof}
By \cite[Theorem~3.5]{Bai08}, we know that $(A\oplus \mathfrak b,\circ_\star)$ is a pre-Lie algebra if and only if Eqs.~\eqref{eq:mppl1}-~\eqref{eq:mppl4} hold. So we only need to check that $P_\star$ satisfies Eq.~\eqref{eq:averop} if and only if Eqs.~\eqref{eq:rapL1}-~\eqref{eq:rapL2} hold. In fact, for all $x,y\in A$ and $ a,b\in\mathfrak b$, we have
\begin{align*}
P_\star(x+a)\circ_\star P_\star(y+b)
&\overset{\eqref{eq:zh2}}{=} \bigl(P_A(x)+P_{\mathfrak b}(a)\bigr)\circ_\star
   \bigl(P_A(y)+P_{\mathfrak b}(b)\bigr) \\
&\overset{\eqref{eq:zh1}}{=} P_A(x)\circ_A P_A(y)
 + \rho_{\mathfrak b} \bigl(P_{\mathfrak b}(a)\bigr) P_A(y)
 + \varphi_{\mathfrak b} \bigl(P_{\mathfrak b}(b)\bigr) P_A(x) \\
&\quad
 + P_{\mathfrak b}(a)\circ_{\mathfrak b} P_{\mathfrak b}(b)
 + \rho_A \bigl(P_A(x)\bigr) P_{\mathfrak b}(b)
 + \varphi_A \bigl(P_A(y)\bigr) P_{\mathfrak b}(a),\\
P_\star\bigl(P_\star(x+a)\circ_\star (y+b)\bigr)
&\overset{\eqref{eq:zh2}}{=} P_\star\Bigl(\bigl(P_A(x)+P_{\mathfrak b}(a)\bigr)\circ_\star (y+b)\Bigr) \\
&\overset{\eqref{eq:zh1}}{=} P_\star\Bigl(
   P_A(x)\circ_A y
 + \rho_{\mathfrak b} \bigl(P_{\mathfrak b}(a)\bigr)y
 + \varphi_{\mathfrak b}(b)P_A(x) \\
&\qquad\qquad
 + P_{\mathfrak b}(a)\circ_{\mathfrak b} b
 + \rho_A \bigl(P_A(x)\bigr)b
 + \varphi_A(y)P_{\mathfrak b}(a)
 \Bigr) \\
&\overset{\eqref{eq:zh2}}{=} P_A\bigl(P_A(x)\circ_A y\bigr)
 + P_A\bigl(\rho_{\mathfrak b}(P_{\mathfrak b}(a))y\bigr)
 + P_A\bigl(\varphi_{\mathfrak b}(b)P_A(x)\bigr) \\
&\quad
 + P_{\mathfrak b}\bigl(P_{\mathfrak b}(a)\circ_{\mathfrak b} b\bigr)
 + P_{\mathfrak b}\bigl(\rho_A(P_A(x))b\bigr)
 + P_{\mathfrak b}\bigl(\varphi_A(y)P_{\mathfrak b}(a)\bigr),\\
P_\star\bigl((x+a)\circ_\star P_\star(y+b)\bigr)
&\overset{\eqref{eq:zh2}}{=} P_\star\Bigl((x+a)\circ_\star\bigl(P_A(y)+P_{\mathfrak b}(b)\bigr)\Bigr) \\
&\overset{\eqref{eq:zh1}}{=} P_\star\Bigl(
   x\circ_A P_A(y)
 + \rho_{\mathfrak b}(a)P_A(y)
 + \varphi_{\mathfrak b}\bigl(P_{\mathfrak b}(b)\bigr)x \\
&\qquad\qquad
 + a\circ_{\mathfrak b} P_{\mathfrak b}(b)
 + \rho_A(x)P_{\mathfrak b}(b)
 + \varphi_A \bigl(P_A(y)\bigr)a
 \Bigr) \\
&\overset{\eqref{eq:zh2}}{=} P_A\bigl(x\circ_A P_A(y)\bigr)
 + P_A\bigl(\rho_{\mathfrak b}(a)P_A(y)\bigr)
 + P_A\bigl(\varphi_{\mathfrak b}(P_{\mathfrak b}(b))x\bigr) \\
&\quad
 + P_{\mathfrak b}\bigl(a\circ_{\mathfrak b} P_{\mathfrak b}(b)\bigr)
 + P_{\mathfrak b}\bigl(\rho_A(x)P_{\mathfrak b}(b)\bigr)
 + P_{\mathfrak b}\bigl(\varphi_A(P_A(y))a\bigr).
\end{align*}
Then we finish the proof by comparing the three equalities above.
\end{proof}

\begin{remark}
If $\Bigl(\bigl((A,\circ_A),P_A\bigr)$, $\bigl((\mathfrak b,\circ_{\mathfrak b}),P_{\mathfrak b}\bigr)$,
$\rho_A,\varphi_A,\rho_{\mathfrak b},\varphi_{\mathfrak b}\Bigr)$ is a matched pair of averaging
pre-Lie algebras $\bigl((A,\circ_A),P_A\bigr)$ and $\bigl((\mathfrak b,\circ_{\mathfrak b}),P_{\mathfrak b}\bigr)$, then we have an averaging
pre-Lie algebra $(A\bowtie \mathfrak b,\,P_A\oplus P_{\mathfrak b})$.
\end{remark}

Let $\Bigl(\bigl((A,\circ_A),P_A\bigr)$, $\bigl((\mathfrak b,\circ_{\mathfrak b}),P_{\mathfrak b}\bigr)$,
$\rho_A,\varphi_A,\rho_{\mathfrak b},\varphi_{\mathfrak b}\Bigr)$ be a matched pair of averaging pre-Lie algebras. Then there are three induced Leibniz algebras $A_{P_A}$, $\mathfrak{b}_{P_{\mathfrak b}}$ and $(A\bowtie\mathfrak{b})_{P_A\oplus P_{\mathfrak b}}$ coming from the three averaging operators $P_A:A\to A$, $P_{\mathfrak b}:\mathfrak{b}\to\mathfrak{b}$ and $P_A\oplus P_{\mathfrak b}:A \oplus\mathfrak{b}\to A \oplus\mathfrak{b}$, respectively. In the following we will show that a matched pair of averaging pre-Lie algebras gives rise to an induced matched pair of Leibniz algebras.

\begin{defn}\cite{AM13}
Let $(\mathcal{G}_1,[\cdot,\cdot]_{\mathcal{G}_1})$ and $(\mathcal{G}_2,[\cdot,\cdot]_{\mathcal{G}_2})$ be two Leibniz algebras. A \textbf{matched pair of Leibniz algebras} is a tuple $(\mathcal{G}_1,\mathcal{G}_2;(\rho^L,\rho^R),(\mu^L,\mu^R))$, where
\begin{enumerate}
\item $(\mathcal{G}_2; \rho^L, \rho^R)$ is a representation of $(\mathcal{G}_1,[\cdot,\cdot]_{\mathcal{G}_1})$;
\item $(\mathcal{G}_1; \mu^L, \mu^R)$ is a representation of $(\mathcal{G}_2,[\cdot,\cdot]_{\mathcal{G}_2})$;
\item for all $x,y\in\mathcal{G}_1$ and $\xi,\eta\in\mathcal{G}_2$, the following identities hold:
\begin{align*}
\rho^R(x)[\xi,\eta]_{\mathcal{G}_2}
&- [\xi,\rho^R(x)\eta]_{\mathcal{G}_2} + [\eta,\rho^R(x)\xi]_{\mathcal{G}_2} - \rho^R(\mu^L(\eta)x)\xi + \rho^R(\mu^L(\xi)x)\eta =0,
\\
\rho^L(x)[\xi,\eta]_{\mathcal{G}_2}
&- [\rho^L(x)\xi,\eta]_{\mathcal{G}_2} - [\xi,\rho^L(x)\eta]_{\mathcal{G}_2} - \rho^L(\mu^R(\xi)x)\eta - \rho^R(\mu^R(\eta)x)\xi =0,
\\
[\rho^L(x)\xi,\eta]_{\mathcal{G}_2}
&+ \rho^L(\mu^R(\xi)x)\eta + [\rho^R(x)\xi,\eta]_{\mathcal{G}_2} + \rho^L(\mu^L(\xi)x)\eta =0,
\\
\mu^R(\xi)[x,y]_{\mathcal{G}_1}
&- [x,\mu^R(\xi)y]_{\mathcal{G}_1} + [y,\mu^R(\xi)x]_{\mathcal{G}_1} - \mu^R(\rho^L(y)\xi)x + \mu^R(\rho^L(x)\xi)y =0,
\\
\mu^L(\xi)[x,y]_{\mathcal{G}_1}
&- [\mu^L(\xi)x,y]_{\mathcal{G}_1} - [x,\mu^L(\xi)y]_{\mathcal{G}_1} - \mu^L(\rho^R(x)\xi)y - \mu^R(\rho^R(y)\xi)x =0,
\\
[\mu^L(\xi)x,y]_{\mathcal{G}_1}
&+ \mu^L(\rho^R(x)\xi)y + [\mu^R(\xi)x,y]_{\mathcal{G}_1} + \mu^L(\rho^L(x)\xi)y =0,
\end{align*}
\end{enumerate}
\end{defn}

\begin{prop}\cite{AM13}\label{prop:mpolbniz}
Let $(\mathcal{G}_1, \mathcal{G}_2;(\rho^L,\rho^R),(\mu^L,\mu^R))$ be a matched pair of Leibniz algebras. Then there is a Leibniz algebra structure on $\mathcal{G}_1\oplus\mathcal{G}_2$ defined by
\begin{align*}
[x+\xi,y+\eta]
&=[x,y]_{\mathcal{G}_1} + \mu^L(\xi)y + \mu^R(\eta)x +[\xi,\eta]_{\mathcal{G}_2} +\rho^L(x)\eta + \rho^R(y)\xi,
\end{align*}
for all $x,y\in \mathcal{G}_1$ and $\xi,\eta\in \mathcal{G}_2$.

Conversely, if $(\mathcal{G}_1\oplus\mathcal{G}_2,[\cdot,\cdot])$ is a Leibniz algebra such that $\mathcal{G}_1$ and $\mathcal{G}_2$ are Leibniz subalgebras, then $(\mathcal{G}_1,\mathcal{G}_2;(\rho^L,\rho^R),(\mu^L,\mu^R))$ is a matched pair of Leibniz algebras, where the representations $(\rho^L,\rho^R)$ of $\mathcal{G}_1$ on $\mathcal{G}_2$ and
$(\mu^L,\mu^R)$ of $\mathcal{G}_2$ on $\mathcal{G}_1$ are determined by
\begin{align*}
[x,\xi]=\rho^L(x)\xi+\mu^R(\xi)x,
\qquad
[\xi,x]=\rho^R(x)\xi+\mu^L(\xi)x,
\end{align*}
for all $x\in\mathcal{G}_1$ and $\xi\in\mathcal{G}_2$.
\end{prop}

\begin{theorem}\label{theorem:mpoplinducempolb}
Let $\Bigl(\bigl((A,\circ_A),P_A\bigr)$, $\bigl((\mathfrak b,\circ_{\mathfrak b}),P_{\mathfrak b}\bigr)$,
$\rho_A,\varphi_A,\rho_{\mathfrak b},\varphi_{\mathfrak b}\Bigr)$ be a matched pair of averaging pre-Lie algebras, and $A_{P_A}$, $\mathfrak{b}_{P_{\mathfrak b}}$ the induced Leibniz algebras. Then
\[
\bigl(A_{P_A},\mathfrak{b}_{P_{\mathfrak b}};
(\rho^L_{(P_A,P_{\mathfrak b})},\rho^R_{(P_A,P_{\mathfrak b})}),(\mu^L_{(P_A,P_{\mathfrak b})},\mu^R_{(P_A,P_{\mathfrak b})})\bigr)
\]
is a matched pair of Leibniz algebras, where the representations
$(\rho^L_{(P_A,P_{\mathfrak b})},\rho^R_{(P_A,P_{\mathfrak b})})$ and $(\mu^L_{(P_A,P_{\mathfrak b})},\mu^R_{(P_A,P_{\mathfrak b})})$ are given by
\begin{align*}
\rho^L_{(P_A,P_{\mathfrak b})}(x)\xi = \rho_A(P_A(x))\xi - \varphi_A(P_A(x))\xi,
\qquad
\rho^R_{(P_A,P_{\mathfrak b})}(x)\xi = -\rho_A(x)(P_{\mathfrak b}\xi) + \varphi_A(x)(P_{\mathfrak b}\xi) ,
\\
\mu^L_{(P_A,P_{\mathfrak b})}(\xi)x = \rho_{\mathfrak b}(P_{\mathfrak b}(\xi))x - \varphi_{\mathfrak b}(P_{\mathfrak b}(\xi))x,
\qquad
\mu^R_{(P_A,P_{\mathfrak b})}(\xi)x = -\rho_{\mathfrak b}(\xi)(P_Ax) + \varphi_{\mathfrak b}(\xi)(P_A\xi),
\end{align*}
for all $x\in A$ and $\xi\in\mathfrak{b}$. Moreover, we have\[
A_{P_A} \bowtie \mathfrak{b}_{P_{\mathfrak b}}
=
\bigl(A \bowtie \mathfrak{b}\bigr)_{P_A \oplus P_{\mathfrak b}} .
\]
as Leibniz algebras.
The matched pair
$
\bigl(A_{P_A},\mathfrak{b}_{P_{\mathfrak b}};
(\rho^L_{(P_A,P_{\mathfrak b})},\rho^R_{(P_A,P_{\mathfrak b})}),(\mu^L_{(P_A,P_{\mathfrak b})},\mu^R_{(P_A,P_{\mathfrak b})})\bigr)
$
is called the \textbf{induced matched pair of Leibniz algebras}.
\end{theorem}
\begin{proof}
By Proposition~\ref{prop:mpopl}, $\bigl((A\oplus \mathfrak b,\circ_\star),P_\star\bigr)$ is an averaging pre-Lie algebra. By Proposition~\ref{prop:ilbniz}, there is an induced Leibniz algebra on $A\oplus \mathfrak b$, denoted by $(A\bowtie \mathfrak b)_{P_\star}$, which contains $A_{P_A}$ and $A_{P_{\mathfrak b}}$ as Leibniz subalgebras. Furthermore, for $x\in A$ and $\xi\in\mathfrak{b}$, we have
\begin{align*}
[x,\xi]_{(A\bowtie \mathfrak{b})_{P_\star}}
&= P_\star(x)\circ_{\star}\xi-\xi\circ_{\star}P_\star(x)\\
&= P_A(x)\circ_{\star}\xi-\xi\circ_{\star}P_A(x)\\
&= \varphi_{\mathfrak{b}}(\xi)P_A(x)+\rho_A(P_A(x))\xi
   -\rho_{\mathfrak{b}}(\xi)P_A(x)-\varphi_A(P_A(x))\xi\\
&= \rho^L_{(P_A,P_{\mathfrak{b}})}(x)\xi
   +\mu^R_{(P_A,P_{\mathfrak{b}})}(\xi)x,\\
[\xi,x]_{(A\bowtie \mathfrak{b})_{P_\star}}
&= P_\star(\xi)\circ_{\star}x-x\circ_{\star}P_\star(\xi)\\
&= P_{\mathfrak{b}}(\xi)\circ_{\star}x-x\circ_{\star}P_{\mathfrak{b}}(\xi)\\
&= \rho_{\mathfrak{b}}(P_{\mathfrak{b}}(\xi))(x)
   +\varphi_A(x)P_{\mathfrak{b}}(\xi)
   -\varphi_{\mathfrak{b}}(P_{\mathfrak{b}}(\xi))(x)
   -\rho_A(x)P_{\mathfrak{b}}(\xi)\\
&= \rho^R_{(P_A,P_{\mathfrak{b}})}(x)\xi
   +\mu^L_{(P_A,P_{\mathfrak{b}})}(\xi)x.
\end{align*}
Then by Proposition~\ref{prop:mpolbniz},
$
\bigl(A_{P_A},\mathfrak{b}_{P_{\mathfrak b}};
(\rho^L_{(P_A,P_{\mathfrak b})},\rho^R_{(P_A,P_{\mathfrak b})}),(\mu^L_{(P_A,P_{\mathfrak b})},\mu^R_{(P_A,P_{\mathfrak b})})\bigr)
$
forms a matched pair of Leibniz algebras.
\end{proof}

\subsection{Manin triple of averaging pre-Lie algebras}
In this subsection, we introduce the notion of quadratic averaging pre-Lie algebras and the concept of Manin triples of averaging pre-Lie algebras. Moreover, we show that a quadratic averaging pre-Lie algebra gives rise to an isomorphism from the regular representation to the coregular representation.

\begin{defn}\cite{WBLS24}\label{def:quadratic-pre-Lie}
Let $(A,\circ)$ be a pre-Lie algebra and $\omega\in\wedge^{2}A^{*}$ a nondegenerate skew-symmetric bilinear form. If $\omega$ is invariant, i.e.,
\begin{equation}
\omega(x \circ y, z)
+\omega\bigl(y,x \circ z - z \circ x\bigr)=0,
\qquad \forall\,x,y,z\in\mathfrak{g},  \mlabel{eq:oinva}
\end{equation}
then $\bigl((A,\circ),\omega\bigr)$ is called a \textbf{quadratic pre-Lie algebra}.
\end{defn}

\begin{defn}\cite{WBLS24}\label{def:Manin-triple-pre-Lie}
Let $(A_1,\circ_1)$ and $(A_2,\circ_2)$ be two pre-Lie algebras. A \textbf{Manin triple of pre-Lie algebras} is a triple $\bigl((\mathcal{A},\circ_{\mathcal{A}},\omega),A_1,A_2\bigr)$, where
\begin{enumerate}
\item $\bigl((\mathcal{A},\circ_{\mathcal{A}} ),\omega \bigr)$ is an even dimensional quadratic pre-Lie algebra;
\item $A_1$ and $A_2$ are pre-Lie subalgebras, both isotropic with respect to $\omega$;
\item $\mathcal{A} =A_1 \oplus A_2$ as vector spaces.
\end{enumerate}
\end{defn}

Now, we introduce the notion of quadratic averaging pre-Lie algebras.
\begin{defn}\label{def:quadratic-averaging-pre-Lie}
Let $\bigl((A,\circ),P\bigr)$ be an averaging pre-Lie algebra, and $\omega \in \wedge^{2}A^{*}$ a nondegenerate skew-symmetric bilinear form.
The triple $\bigl((A,\circ),P,\omega\bigr)$ is called a \textbf{quadratic averaging pre-Lie algebra}
if $\bigl((A,\circ),\omega\bigr)$ is a quadratic pre-Lie algebra and the following compatibility condition holds:
\begin{equation}
\omega(Px,y)-\omega(x,Py)=0,
\qquad \forall\,x,y\in A.  \mlabel{eq:quavecomp}
\end{equation}
\end{defn}

\begin{theorem}\label{theorem:omegasharp}
If $\bigl((A,\circ),P,\omega\bigr)$ is a quadratic averaging pre-Lie algebra, then the linear map
\[
\omega^\sharp : A \longrightarrow A^*
\]
defined by
\[
\langle \omega^\sharp(x), y \rangle = \omega(x,y), \qquad \forall\, x,y \in A,
\]
is an isomorphism from the regular representation$\bigl((A,L,R),P\bigr)$ to the coregular representation
$\bigl((A^*,L^*-R^*,-R^*),P^*\bigr).$
\end{theorem}
\begin{proof}
For $\forall\, x,y, z \in A$, we have
\begin{align*}
\bigl\langle (L^* - R^*)(x)\,\omega^\sharp(y) -  \omega^\sharp(L(x)y),\, z \bigr\rangle
&= \bigl\langle (L^* - R^*)(x)\,\omega^\sharp(y),\, z \bigr\rangle
   - \bigl\langle \omega^\sharp(x \circ y),\, z \bigr\rangle \\
&= - \bigl\langle \omega^\sharp(y),\, (L - R)(x)z \bigr\rangle
   - \bigl\langle \omega^\sharp(x \circ y),\, z \bigr\rangle \\
&= - \bigl\langle \omega^\sharp(y),\, x \circ z - z \circ x \bigr\rangle
   - \bigl\langle \omega^\sharp(x \circ y),\, z \bigr\rangle \\
&= - \omega\bigl(y,\, x \circ z - z \circ x \bigr)
   - \omega(x \circ y,\, z) \\
&\overset{\eqref{eq:oinva}}{=} 0 ,
\end{align*}
\begin{align*}
\bigl\langle -R^*(x)\,\omega^\sharp(y) - \omega^\sharp(R(x)y),\, z \bigr\rangle
&= \bigl\langle -R^*(x)\,\omega^\sharp(y),\, z \bigr\rangle
   - \bigl\langle \omega^\sharp(R(x)y),\, z \bigr\rangle \\
&= \bigl\langle \omega^\sharp(y),\, R(x)z \bigr\rangle
   - \bigl\langle \omega^\sharp(y \circ x),\, z \bigr\rangle \\
&= \bigl\langle \omega^\sharp(y),\, z \circ x \bigr\rangle
   - \bigl\langle \omega^\sharp(y \circ x),\, z \bigr\rangle \\
&= \omega(y,\, z \circ x) - \omega(y \circ x,\, z) \\
&\overset{\eqref{eq:oinva}}{=} \omega(x \circ y,\, z) + \omega(y,\, x \circ z) - \omega(y \circ x,\, z) \\
&= \omega(x \circ y - y \circ x,\, z)
   + \omega(y,\, x \circ z) \\
& \xlongequal{\text{skew-symmetric}}- \Bigl( \omega(x \circ z,\, y)
   + \omega(z,\, x \circ y - y \circ x) \Bigr) \\
&\overset{\eqref{eq:oinva}}{=} 0 ,
\end{align*}
\begin{align*}
\bigl\langle \omega^\sharp(P(x)) - P^*(\omega^\sharp(x)),\, z \bigr\rangle
= \bigl\langle \omega^\sharp(P(x)),\, z \bigr\rangle
   - \bigl\langle \omega^\sharp(x),\, P(z) \bigr\rangle
= \omega(P(x),\, z) - \omega(x,\, P(z))
\overset{\eqref{eq:quavecomp}}{=} 0 ,
\end{align*}
which imply that Eq.~\eqref{eq:homocondi} hold.

Therefore, $\omega^\sharp : A \longrightarrow A^*$ is an isomorphism from the regular representation$\bigl((A,L,R),P\bigr)$ to the coregular representation $\bigl((A^*,L^*-R^*,-R^*),P^*\bigr).$
\end{proof}

Now we introduce the notion of Manin triples of averaging pre-Lie algebras by using quadratic averaging pre-Lie algebras given in Definition~\ref{def:quadratic-averaging-pre-Lie}.

\begin{defn}\label{def:Manin-triple-aver-pre-Lie}
A \textbf{Manin triple of averaging pre-Lie algebras} is a triple
\[
\Bigl(
\bigl((\mathcal{A},\circ_{\mathcal{A}}),\mathcal{P},\omega\bigr),
\bigl((A,\circ_A),P_A\bigr),
\bigl((\mathfrak{b},\circ_{\mathfrak{b}}),P_{\mathfrak{b}}\bigr)
\Bigr),
\]
where
\begin{enumerate}
\item $\bigl((\mathcal{A},\circ_{\mathcal{A}}),\mathcal{P},\omega\bigr)$ is an even dimensional quadratic averaging pre-Lie algebra;
\item $\bigl((A,\circ_A),P_A\bigr)$ and $\bigl((\mathfrak{b},\circ_{\mathfrak{b}}),P_{\mathfrak{b}}\bigr)$ are averaging pre-Lie subalgebras, i.e.\ $A$ and $\mathfrak{b}$ are pre-Lie subalgebras of $\mathcal{A}$ and $\mathcal{P}|_{A}=P_A$,
    $\mathcal{P}|_{\mathfrak{b}}=P_{\mathfrak{b}}$;
\item both $A$ and $\mathfrak{b}$ are isotropic with respect to $\omega$;
\item $\mathcal{A}=A\oplus\mathfrak{b}$ as vector spaces.
\end{enumerate}
\end{defn}

Obviously, we have the following characterization of Manin triple of averaging pre-Lie algebras.

\begin{prop}
A triple $\Bigl(
\bigl((\mathcal{A},\circ_{\mathcal{A}}),\mathcal{P},\omega\bigr),
\bigl((A,\circ_A),P_A\bigr),
\bigl((\mathfrak{b},\circ_{\mathfrak{b}}),P_{\mathfrak{b}}\bigr)
\Bigr)$,
where $\bigl((\mathcal{A},\circ_{\mathcal{A}}),\mathcal{P},\omega\bigr)$ is a quadratic averaging pre-Lie algebra, $\bigl((A,\circ_A),P_A\bigr)$ and $\bigl((\mathfrak{b},\circ_{\mathfrak{b}}),P_{\mathfrak{b}}\bigr)$ are averaging pre-Lie algebras, is a Manin triple of averaging pre-Lie algebras
if and only if $\Bigl(
\bigl((\mathcal{A},\circ_{\mathcal{A}}),\omega\bigr),
(A,\circ_A),(\mathfrak{b},\circ_{\mathfrak{b}})
\Bigr)$ is a Manin triple of pre-Lie algebras
such that
\[
\mathcal{P}\big|_A=P_A,\qquad \mathcal{P}\big|_{\mathfrak{b}}=P_{\mathfrak{b}}.
\]
\end{prop}

\section{Averaging pre-Lie bialgebras and averaging operators on quadratic Rota-Baxter pre-Lie algebras} \label{sec:averbiandaverop}

In this section, we introduce the definitions of averaging pre-Lie bialgebras and that of averaging operators on quadratic Rota-Baxter pre-Lie algebras. Moreover, we show that a quadratic Rota-Baxter pre-Lie algebra and its averaging operator gives rise to an averaging pre-Lie bialgebra.

\subsection{Averaging pre-Lie bialgebras}

In this subsection, we introduce the notion of averaging pre-Lie bialgebras. Moreover, under certain conditions, we show that matched pairs of averaging pre-Lie algebras, Manin triples of averaging pre-Lie algebras, and averaging pre-Lie bialgebras are equivalent.

\begin{defn}
An \textbf{averaging pre-Lie coalgebra} is a pair $\big( (A,\Delta), S\big)$, where $(A,\Delta)$ is a pre-Lie coalgebra and $S : A \to A$ is a linear map such that
\begin{equation}
(S \otimes S)\Delta(x) = (S \otimes \mathrm{id}_A)\Delta(S(x)) = (\mathrm{id}_A \otimes S)\Delta(S(x)),
\qquad \forall x \in A.  \mlabel{eq:aapLieco}
\end{equation}
\end{defn}

\begin{prop}\label{prop:dual-averaging-preLie}
Let $\bigl((A,\Delta),S\bigr)$ be an averaging pre-Lie coalgebra. Define a linear map $S^{*}:A^{*}\to A^{*}$ by
\begin{align*}
\langle S^{*}(\xi),x\rangle=\langle \xi,S(x)\rangle,\qquad \forall\,\xi\in A^{*},\ x\in A, \mlabel{eq:def-S-star}
\end{align*}
and define a bilinear map $\circ_{A^*}:A^{*}\otimes A^{*}\to A^{*}$ by
\begin{equation}
\langle \xi \circ_{A^*} \eta, x\rangle=\langle \xi\otimes\eta,\Delta(x)\rangle,
\qquad \forall\,\xi,\eta\in A^{*},\ x\in A. \mlabel{eq:def-circ-star}
\end{equation}
Then $\bigl((A^{*},\circ_{A^*}),S^{*}\bigr)$ is an averaging pre-Lie algebra.
\end{prop}
\begin{proof}
It's a direct calculation.
\end{proof}

\begin{defn} \label{def:apLbi}
An \textbf{averaging pre-Lie bialgebra} is a quintuple $(A,\circ,\Delta,P,S)$, where
\begin{enumerate}
\item $\big((A,\circ),P\big)$ is an averaging pre-Lie algebra;
\item $\big((A,\Delta),S\big)$ is an averaging pre-Lie coalgebra;
\item $(A,\circ,\Delta)$ is a pre-Lie bialgebra;
\item $\big((A,\circ),P\big)$ is $S$-admissible, i.e.\ Eqs.~\eqref{eq:regu1} and \eqref{eq:regu2} hold;  \mlabel{eq:apLbid}
\item $\big((A^*,\Delta^* := \circ_{A^*}),S^*\big)$ is $P^*$-admissible, i.e.\ for all $x\in A$,  \mlabel{eq:apLbie}
\begin{equation}
(P \otimes S)\Delta(x) = (\mathrm{id}_A \otimes S)\Delta(P(x)) = (P \otimes \mathrm{id}_A)\Delta(P(x)), \mlabel{eq:aplb1}
\end{equation}
\begin{equation}
(S \otimes P)\Delta(x) = (S \otimes \mathrm{id}_A)\Delta(P(x)) = (\mathrm{id}_A \otimes P)\Delta(P(x)). \mlabel{eq:aplb2}
\end{equation}
\end{enumerate}
\end{defn}

\begin{coro}
If $(A,\circ,\Delta,P,P)$ is an averaging pre-Lie bialgebra, then $(A^*,\circ_{A^*},\Delta_{A^*},P^*,P^*)$ is an averaging pre-Lie bialgebra where $\circ_{A^*}$ is given by Eq.~\eqref{eq:def-circ-star} and $\Delta_{A^*}: A^* \to A^* \otimes A^* $ is given by the following equation
\begin{align*}
\langle \Delta_{A^*}(\xi) , x \otimes y\rangle = \langle \xi ,x \circ y\rangle,
\qquad \forall \xi\in A^{*},\ x, y\in A.
\end{align*}
\end{coro}
\begin{proof}
It can be verified directly from Definition ~\ref{def:apLbi}.
\end{proof}

The following theorem provides an equivalent characterization of averaging pre-Lie bialgebras via matched pairs of averaging pre-Lie algebras.

\begin{theorem} \label{theorem:mpoplbi}
Let $\bigl((A,\circ),P\bigr)$ and $\bigl((A^{*},\circ_{A^*}),S^{*}\bigr)$ be averaging pre-Lie algebras.
Then the quadruple $(A,\circ,\Delta,P,S)$ is an averaging pre-Lie bialgebra where $\Delta$ is given by Eq.~\eqref{eq:def-circ-star} if and only if
\[
\Big(\bigl((A,\circ),P\bigr),\;
\bigl((A^{*},\circ_{A^*}),S^{*}\bigr),\;
L_{\circ}^{*}-R_{\circ}^{*},\;
-\,R_{\circ}^{*},\;
L_{\circ_{A^*}}^{*}-R_{\circ_{A^*}}^{*},\;
-\,R_{\circ_{A^*}}^{*}\Big)
\]
is a matched pair of averaging pre-Lie algebras
$\bigl((A,\circ),P\bigr)$ and $\bigl((A^{*},\circ_{A^*}),S^{*}\bigr)$.
\begin{proof}
According to \cite[Proposition~4.2]{Bai08}, we find that $(A,\circ,\Delta)$ is a
pre-Lie bialgebra if and only if
\[
\Big( \bigl((A,\circ),(A^{*},\circ_{A^*})\bigr),L_{\circ}^{*}-R_{\circ}^{*},\;
-\,R_{\circ}^{*},\;
L_{\circ_{A^*}}^{*}-R_{\circ_{A^*}}^{*},\;
-\,R_{\circ_{A^*}}^{*}
\Big)
\]
is a matched pair of pre-Lie algebras $(A,\circ)$ and $(A^{*},\circ_{A^*})$.
The rest of the argument is obtained by Items ~\eqref{eq:apLbid} and ~\eqref{eq:apLbie} in Definition~\ref{def:apLbi}.
\end{proof}
\end{theorem}

According to ~\cite{WBLS24}, on the one hand, given a pre-Lie bialgebra $(A, \circ, \Delta)$, the triple $\bigl((A\bowtie A^*,\omega),A,A^*\bigr)$ is a Manin triple of pre-Lie algebras, where $\omega$ is given by the following identity
\begin{equation}
\omega(x+\xi,y+\eta)=<\xi,y>-<\eta,x>,
\qquad
\forall\,x,y\in A,\ \xi,\eta\in A^*. \mlabel{eq:Maform}
\end{equation}

On the other hand, given a Manin triple $\bigl((\mathcal{A},\circ_{\mathcal{A}},\omega),A,\mathfrak{b}\bigr)$, identifying
$\mathfrak{b}$ with $A^*$ by using the nondegenerate invariant skew-symmetric bilinear form $\omega$, we can obtain a pre-Lie bialgebra
$(A,\circ, \Delta)$, where $\circ,\Delta$ are defined
respectively by
\begin{align*}
x \circ y = -x \circ_{\mathcal{A}} y,
\qquad
\langle \Delta(x),\xi\otimes\eta\rangle
= \langle x,\xi\circ_{\mathcal{A}}\eta\rangle, \mlabel{eq:Maalphabeta}
\end{align*}
for all $x,y\in A$ and $\xi,\eta\in A^*$.

Similar to the classical case, we have the following result.
\begin{theorem} \label{theorem:mtriandaplbi}
There is a one-to-one correspondence between Manin triple of averaging pre-Lie algebras and averaging pre-Lie bialgebras.
\end{theorem}
\begin{proof}
Let $(A,\circ,\Delta,P,P^*)$ be an averaging pre-Lie bialgebra. Then by Proposition~\ref{prop:dual-averaging-preLie}, $\bigl((A^*,\circ_{A^*}),P^*\bigr)$ is an averaging pre-Lie algebra. From item~(\ref{eq:apLbid}) and item~(\ref{eq:apLbie}) of Definition~\ref{def:apLbi}, we know that $\big((A^*,L_{\circ}^* - R_{\circ}^{*}, -R_{\circ}^{*}),P^*\big)$ is a representation of $\big((A,\circ),P\big)$ and $\big((A,L_{\circ_{A^*}}^{*}-R_{\circ_{A^*}}^{*},-R_{\circ_{A^*}}^{*}),P\big)$ is a representation of $\big((A^*,\circ_{A^*}),P^*\big)$, then by Proposition~\ref{prop:mpopl}, $(A\bowtie A^*,\; P\oplus P^*)$ is an averaging pre-Lie algebra. Moreover, since $\omega$ is given by Eq.~\eqref{eq:Maform}, we have
\begin{align*}
&\quad \omega\bigl((P\oplus P^{*})(x+\xi),\,y+\eta\bigr)
-\omega\bigl(x+\xi,\,(P\oplus P^{*})(y+\eta)\bigr)\\
&= \omega\bigl(P(x)+P^{*}(\xi),\,y+\eta\bigr)
-\omega\bigl(x+\xi,\,P(y)+P^{*}(\eta)\bigr)\\
&\overset{\eqref{eq:Maform}}{=} \langle P^{*}(\xi),y\rangle-\langle \eta,P(x)\rangle
-\langle \xi,P(y)\rangle+\langle P^{*}(\eta),x\rangle\\
&=0,
\end{align*}
which implies that $\bigl((A\bowtie A^*,\; P\oplus P^*),\omega\bigr)$ is a quadratic averaging pre-Lie algebra. Consequently,
$\Bigl(
\bigl((A\bowtie A^*,\; P\oplus P^*),\omega\bigr),
\bigl((A,\circ),P\bigr),
\bigl((A^*,\circ_{A^*}),P^*\bigr)
\Bigr)$
is a Manin triple of averaging pre-Lie algebras.

Conversely, let
$\Bigl(
\bigl((\mathcal{A}=A \bowtie \mathfrak{b},\circ_{\mathcal{A}}),\mathcal{P},\omega\bigr),
\bigl((A,\circ_A),P_A\bigr),
\bigl((\mathfrak{b},\circ_{\mathfrak{b}}),P_{\mathfrak{b}}\bigr)
\Bigr)$
be a Manin triple of averaging pre-Lie algebras. Then, similar to the classical argument, first we can identify $\mathfrak{b}$ with $A^{*}$ by using the nondegenerate invariant skew-symmetric bilinear form $\omega$, and we can obtain a pre-Lie bialgebra $(A,\circ,\Delta)$. Then we can identify $P_{\mathfrak{b}}$ with $P_A^{*}$ by Eq.~\eqref{eq:quavecomp}. Consequently, both $\bigl((A,\circ_A),P_A\bigr)$ and
$\bigl((\mathfrak{b},\circ_{\mathfrak{b}}),P_{\mathfrak{b}}\bigr)$ are averaging pre-Lie algebras, i.e.
$(A,\circ,\Delta,P_A,P_{A}^{*})$ is an averaging pre-Lie bialgebra.
\end{proof}

\begin{coro}
Let $((A,\circ),P)$ and $((A^{*},\circ_{A^*}),P^{*})$ be averaging pre-Lie algebras.
Then the following statements are equivalent:
\begin{enumerate}
  \item The quintuple $(A,\circ,\Delta,P,P^{*})$ is an averaging pre-Lie bialgebra.

  \item $\bigl(((A,\circ),P),((A^{*},\circ^{*}),P^{*}),\,L_{\circ}^{*}-R_{\circ}^{*},\,-R_{\circ}^{*},
      \,L_{\circ_{A^*}}^{*}-R_{\circ_{A^*}}^{*},\,-R_{\circ_{A^*}}^{*}\bigr)$
  is a matched pair of averaging pre-Lie algebras.

  \item $
  \bigl((A\oplus A^{*},\,\circ_{\star},\,P\oplus P^{*},\,\omega),\ ((A,\circ),P),\ ( (A^{*},\circ_{A^*}),P^{*})\bigr)
  $
  is a Manin triple of averaging pre-Lie algebras,
  where $A$ and $A^{*}$ are isotropic averaging pre-Lie subalgebras of $A\bowtie A^{*}$.
\end{enumerate}
\end{coro}
\begin{proof}
It follows directly from Theorems ~\ref{theorem:mpoplbi} and ~\ref{theorem:mtriandaplbi}.
\end{proof}

\subsection{Averaging operators on quadratic Rota-Baxter pre-Lie algebras}

Recall that a pre-Lie bialgebra $(A,\circ,\Delta)$ is called \textbf{coboundary} if there exists an $r \in A\otimes A$ such that
\begin{equation}
\Delta_r(x):= (L_x \otimes \mathrm{id} + \mathrm{id} \otimes (L_x - R_x))r, \quad \forall x \in A. \mlabel{eq:deltar}
\end{equation}

Let $r=\sum_i a_i\otimes b_i \in A \otimes  A$.
Let $U(A)$ be the universal enveloping algebra of the pre-Lie algebra $(A, \circ)$. We introduce $r_{12},\,r_{13},\,r_{23}$ in
$U(A)\otimes U(A)\otimes U(A)$ as follows:
\[
r_{12}=\sum_i a_i\otimes b_i\otimes 1,\qquad
r_{23}=\sum_i 1\otimes a_i\otimes b_i,\qquad
r_{13}=\sum_i a_i\otimes 1\otimes b_i,
\]
and the bilinear map $r_{12}\circ r_{13}$ is defined by
\[
r_{12}\circ r_{13}
=(\sum_i a_i\otimes b_i\otimes 1)\circ (\sum_j a_j\otimes 1\otimes b_j)
=\sum_{i,j}a_i \circ a_j \otimes b_i \otimes b_j,
\]
and similarly for $r_{13}\circ r_{23}$ and $r_{12}\circ r_{23}$.

\begin{theorem}\cite{Bai08}\label{theorem:preliebi}
Let $(A,\circ)$ be a pre-Lie algebra and $r \in A\otimes A$. Writing $r$ as $r=a+\Lambda$ with $a\in \wedge^2 A$ and $\Lambda\in S^2(A)$. Then the map $\Delta_r$ defined by Eq.~\eqref{eq:deltar} makes $(A, \Delta_r)$ into a pre-Lie coalgebra such that $(A, \circ, \Delta_r)$ is a pre-Lie bialgebra if and only if for all $x,y\in A$, the following two conditions are satisfied:
\begin{equation}
(P(x\circ y)-P(x)P(y))(a)=0, \mlabel{eq:symmetaic}
\end{equation}
\begin{equation}
Q(x)[[r,r]]=0, \mlabel{eq:S-equation}
\end{equation}
where
$Q(x)=L_x\otimes \mathrm{id}\otimes \mathrm{id}+\mathrm{id}\otimes L_x\otimes \mathrm{id}+\mathrm{id}\otimes \mathrm{id}\otimes \mathrm{ad}_x$,
$P(x)=L_x\otimes \mathrm{id}+\mathrm{id}\otimes L_x$, and $[[r,r]]$ is defined by
\begin{align*}
[[r,r]]
&=r_{13}\circ r_{12}-r_{23}\circ r_{21} +[r_{23},r_{12}]-[r_{13},r_{21}]-[r_{13},r_{23}]\\
&=\sum_{i,j=1}^n
\bigl(
a_i \otimes b_i \circ a_j \otimes b_j
+ a_i \otimes a_j \otimes b_i \circ b_j
\bigr)
-
\sum_{i,j=1}^n
\bigl(
a_i \circ a_j \otimes b_i \otimes b_j
+ a_i \otimes a_j \otimes b_j \circ b_i
\bigr).
\end{align*}
\end{theorem}
We denote this pre-Lie bialgebra by $(A, \circ, \Delta_r)$.

In particular, the equation $ [[r,r]]=0 $ is called the \textbf{classical Yang-Baxter equation in the pre-Lie algebra $(A, \circ)$} or the \textbf{$\mathbb{S}$-equation in $(A,\circ)$}.

\begin{defn}\cite{WBLS24}
Let $(A,\circ)$ be a pre-Lie algebra. If $r \in A\otimes A$ satisfies $[[r,r]]=0$ and the skew-symmetric $a$ of $r$ satisfies
\begin{align*}
(L_x \otimes \mathrm{id} + \mathrm{id} \otimes \mathrm{ad}_x )a = 0, \quad \forall x\in A.
\end{align*}
Then the pre-Lie bialgebra $(A, \circ, \Delta_r)$ induced by $r$ is called a \textbf{quasi-triangular pre-Lie bialgebra}. Moreover, if $r$ is symmetric, then $(A, \circ, \Delta_r)$ is called a \textbf{triangular pre-Lie bialgebra}.
\end{defn}

Let $r\in A\otimes A$. Define $r_+,r_-: A^*\to\ A$ by
\begin{align*}
\langle r_+(\xi),\eta\rangle = \langle r, \xi \otimes \eta \rangle,
\qquad
\langle r_-(\xi),\eta\rangle = \langle \xi,r_+(\eta)\rangle = \langle r, \eta \otimes \xi \rangle,
\quad
\forall\,\xi,\eta\in A^* .
\end{align*}

\begin{defn}\cite{WBLS24}
A quasi-triangular pre-Lie bialgebra $(A, \circ, \Delta_r)$ is called \textbf{factorizable} if the linear map $I:=r_{+} - r_{-}$ is a linear isomorphism of vector spaces.
\end{defn}

\begin{defn}\cite{WBLS24}
Let $(A,\circ)$ be a pre-Lie algebra. A linear operator $B:A\to A$ is called a \textbf{Rota--Baxter operator of weight $\lambda$} if
\[
Bx \circ By
=
B\Big(Bx \circ y +x \circ By + \lambda x \circ y\Big),
\qquad \forall\,x,y\in A.
\]
Moreover, we call a pre-Lie algebra $(A,\circ)$ with a Rota--Baxter operator $B$ \textbf{Rota--Baxter pre-Lie algebra of weight $\lambda$} and denote it by $(A,\circ,B)$.
\end{defn}

Let $(A,\circ,B)$ be a Rota--Baxter pre-Lie algebra of weight $\lambda$. Then there is a new pre-Lie multiplication $\circ_B$ on $A$ defined by
\begin{equation}
x \circ_B y = Bx \circ y +x \circ By + \lambda x \circ y. \mlabel{eq:indub}
\end{equation}
The pre-Lie algebra $(A,\circ_B)$ is called the \textbf{descendent pre-Lie algebra}, and denoted by $A_B$. It is obvious that $B$ is a pre-Lie algebra homomorphism from $A_B$ to $A$.

\begin{prop}
Let $(A,\circ,B)$ be a Rota--Baxter pre-Lie algebra of weight $\lambda$ and $P:A\to A$ an averaging operator on $A$. If
$P\circ B = B\circ P$, then $P$ is also an averaging operator on the descendent pre-Lie algebra $(A,\circ_B)$.
\end{prop}
\begin{proof}
For all $x, y \in A$, we have
\begin{align*}
P(x)\circ_B P(y) - P\bigl(P(x)\circ_B y\bigr)
&\overset{\eqref{eq:indub}}{=} B(P(x))\circ P(y) + P(x)\circ B(P(y)) + \lambda\, P(x)\circ P(y) \\
&\quad - P\bigl(B(P(x))\circ y + P(x)\circ B(y) + \lambda\, P(x)\circ y\bigr) \\
&= P(B(x))\circ P(y) + P(x)\circ P(B(y)) + \lambda\, P(x)\circ P(y) \\
&\quad - P\bigl(P(B(x))\circ y\bigr) - P\bigl(P(x)\circ B(y)\bigr) - \lambda\, P\bigl(P(x)\circ y\bigr) \\
&\overset{\eqref{eq:averop}}{=} 0.
\end{align*}
Similarly, we get $P(x)\circ_B P(y) = P\bigl(x \circ_B P(y)\bigr)$. Therefore, $P$ is an averaging operator on the descendent pre-Lie algebra $(A,\circ_B)$.
\end{proof}

\begin{defn}\cite{WBLS24}
Let $(A,\circ,B)$ be a Rota-Baxter pre-Lie algebra of weight $\lambda$ and $\bigl((A,\circ),\omega\bigr)$ a quadratic pre-Lie algebra.
Then the quadruple $(A,\circ,B,\omega)$ is called a \textbf{quadratic Rota-Baxter pre-Lie algebra of weight $\lambda$} if the following compatibility condition holds:
\begin{align*}
\omega(Bx,y)+\omega(x,By)+\lambda \omega(x,y)=0, \qquad \forall\, x,y \in A .
\end{align*}
\end{defn}

It is known that quadratic Rota-Baxter pre-Lie algebras of weight 1 are one-to-one correspondence with factorizable pre-Lie bialgebras.

\begin{theorem}\cite{WBLS24}\label{theorem:frac-preliebi}
Let $(A, \circ, \Delta_r)$ be a factorization pre-Lie bialgebra with $I:=r_{+} - r_{-}$. Then $(A, \circ, B_I, \omega_I)$ is a quadratic Rota-Baxter pre-Lie algebra of weight $\lambda$, where the linear map $B_I: A \to A$ and $\omega_I\in\wedge^2 A^*$
are defined respectively by
\begin{align*}
B_I &= \lambda r_-\circ I^{-1},\\
\omega_I(x,y) &= \langle I^{-1}x,\,y\rangle,\qquad \forall\,x,y\in A.
\end{align*}

Conversely, let $(A, \circ,B,\omega)$ be a quadratic Rota-Baxter pre-Lie algebra of weight $\lambda$ $(\lambda\neq 0)$, and $J_\omega:A^*\to A$ the induced linear isomorphism given by $\langle J_\omega^{-1}x,\,y\rangle:=\omega(x,y).$ Then $r^{B, \omega} \in A\otimes A$ defined by
\begin{align*}
r_+^{B, \omega}:=\frac{1}{\lambda}(B+\lambda\,\mathrm{id})\circ J_\omega: A^*\to A,
\qquad
r_+^{B, \omega}(\xi)=r(\xi,\cdot),\qquad \forall\,\xi \in A^*
\end{align*}
satisfies  the classical Yang-Baxter equation in the pre-Lie algebra $(A, \circ)$ and thus gives rise to a factorizable pre-Lie bialgebra $(A, \circ, \Delta_{r^{B, \omega}})$ with a comultiplication $\Delta_{r^{B, \omega}} = (L_x \otimes \mathrm{id} + \mathrm{id} \otimes \mathrm{ad}_x )r^{B, \omega}$. Moreover, for all $x, y \in A$,
\begin{equation}
J_\omega^{-1}x \cdot_r J_\omega^{-1}y = \frac{1}{\lambda}J_\omega^{-1}(x \circ_B y), \mlabel{eq:rb}
\end{equation}
where $\cdot_r$ is defined by:
\begin{align*}
\xi \cdot_r \eta = \mathrm{ad}^*_{r_+(\xi)}\eta - R^*_{r_-(\eta)}\xi, \qquad  \forall \xi, \eta \in A^*.
\end{align*}
\end{theorem}

Now we introduce the notion of an averaging operator on a quadratic Rota-Baxter pre-Lie algebra, and show that an averaging operator on a quadratic Rota-Baxter pre-Lie algebra naturally induces an averaging pre-Lie bialgebra.
To avoid an ambiguity, in what follows we denote by $P^{*}$ the dual map defined by $\langle P^{*}f, x\rangle = \langle f, Px\rangle$, and by $P^{*,\omega}$ the adjoint map defined by $\omega(Pa, b) = \omega(a, P^{*,\omega} b)$.

\begin{defn}
Let $(A, \circ, B, \omega)$ be a quadratic Rota-Baxter pre-Lie algebra of weight $\lambda$. A linear map $P:A\to A$ is called an \textbf{averaging operator} on $(A, \circ, B, \omega)$, if $B\circ P^{*,\omega} =-\,P\circ B$ and $\bigl((A, \circ), P\bigr)$ is an averaging pre-Lie algebra.
\end{defn}

\begin{remark}
If $(A, \circ, B, \omega)$ be a quadratic Rota-Baxter pre-Lie algebra of weight $\lambda$, then the compatibility condition $B\circ P^{*,\omega}=-P\circ B$ implies the identity $\lambda\,(P+P^{*,\omega})=0.$ Consequently, if $\lambda\neq 0$, we get $P = -P^{*,\omega}$.
\end{remark}

\begin{theorem}
Let $(A, \circ,B,\omega)$ be a quadratic Rota-Baxter pre-Lie algebra of weight $\lambda$ and $P:A\to A$ be an averaging operator on $(A, \circ,B,\omega)$. Then $-P^{*,\omega}$ is an averaging operator on the descendent pre-Lie algebra $A_B$.
\end{theorem}
\begin{proof}
Since $(A, \circ,B,\omega)$ be a quadratic Rota-Baxter pre-Lie algebra of weight $\lambda$, then the adjoint map $B^{*} = -B - \lambda\,\mathrm{id}$
is also a Rota-Baxter operator of weight $\lambda$. For all $x, y \in A$, we have
\begin{align*}
(-P^{*,\omega}x)\circ_{B}(-P^{*,\omega}y)
&= (P^{*,\omega}x)\circ_{B}(P^{*,\omega}y) \\
&\overset{\eqref{eq:indub}}{=} B(P^{*,\omega}x)\circ P^{*,\omega}y
   + P^{*,\omega}x\circ B(P^{*,\omega}y)
   + \lambda\, P^{*,\omega}x\circ P^{*,\omega}y \\
&= B(P^{*,\omega}x)\circ P^{*,\omega}y
   - P^{*,\omega}x\circ B^{*}(P^{*,\omega}y).
\end{align*}
Recall that if $\lambda \neq 0$, then $P^{*,\omega} = -P$. It means that for any weight $\lambda$, the map $B^{*}$ also satisfies the condition
\[
P \circ B^{*} = -\, B^{*} \circ P^{*,\omega}.
\]
Consider $\forall z \in A$. Then
\begin{align*}
\omega\bigl(B(P^{*,\omega}x)\circ P^{*,\omega}y,\, z\bigr)
&\overset{\eqref{eq:oinva}}{=} -\,\omega\bigl(P^{*,\omega}y,\; B(P^{*,\omega}x)\circ z - z\circ B(P^{*,\omega}x)\bigr) \\
&= -\,\omega\bigl(y,\; P\bigl(B(P^{*,\omega}x)\circ z - z\circ B(P^{*,\omega}x)\bigr)\bigr) \\
&= \omega\bigl(y,\; P\bigl(P(Bx)\circ z - z\circ P(Bx)\bigr)\bigr) \\
&\overset{\eqref{eq:averop}}{=} \omega\bigl(y,\; P(Bx)\circ P(z) - P(z)\circ P(Bx)\bigr) \\
&\overset{\eqref{eq:oinva}}{=} -\,\omega\bigl(P(Bx)\circ y,\, P(z)\bigr) \\
&= -\,\omega\bigl(P^{*,\omega}(P(Bx)\circ y),\, z\bigr) \\
&= \omega\bigl(-P^{*,\omega}(P(Bx)\circ y),\, z\bigr) \\
&= \omega\bigl(P^{*,\omega}(B(P^{*}x)\circ y),\, z\bigr),
\end{align*}
and
\begin{align*}
\omega\bigl(P^{*,\omega}x \circ B^{*}(P^{*,\omega}y),\, z\bigr)
&\overset{\eqref{eq:oinva}}{=} -\,\omega\bigl(B^{*}(P^{*,\omega}y),\; P^{*,\omega}x \circ z - z \circ P^{*,\omega}x\bigr) \\
&= \omega\bigl(P(B^{*}y),\; P^{*,\omega}x \circ z - z \circ P^{*,\omega}x\bigr) \\
&= \omega\bigl(B^{*}y,\; P^{*,\omega}\bigl(P^{*,\omega}x \circ z - z \circ P^{*,\omega}x\bigr)\bigr) \\
&= \omega\bigl(B^{*}y,\; P(x)\circ P(z) - P(z)\circ P(x)\bigr) \\
&= -\,\omega\bigl(P(x)\circ B^{*}y,\, P(z)\bigr) \\
&= -\,\omega\bigl(P^{*,\omega}(P(x)\circ B^{*}y),\, z\bigr) \\
&= \omega\bigl(P^{*,\omega}\bigl(P(x)\circ By + \lambda\, P(x)\circ y\bigr),\, z\bigr)\\
&= -\,\omega\bigl(P^{*,\omega}\bigl(P^{*,\omega}(x)\circ By + \lambda\, P^{*,\omega}(x)\circ y\bigr),\, z\bigr).
\end{align*}
Therefore,
\begin{align*}
(-P^{*,\omega}x)\circ_{B}(-P^{*,\omega}y)
&= B(P^{*,\omega}x)\circ P^{*,\omega}y - P^{*,\omega}x\circ B^{*}(P^{*,\omega}y) \\
&= P^{*,\omega}\bigl(B(P^{*,\omega}x)\circ y\bigr)
   + P^{*,\omega}\bigl(P^{*,\omega}x\circ B y + \lambda\, P^{*,\omega}x\circ y\bigr) \\
&= -\,P^{*,\omega}\Bigl(B(-P^{*,\omega}x)\circ y
   - P^{*,\omega}x\circ B y - \lambda\, P^{*,\omega}x\circ y\Bigr) \\
&\overset{\eqref{eq:indub}}{=} -\,P^{*,\omega}\bigl((-P^{*,\omega}x)\circ_{B} y\bigr).
\end{align*}
Similarly, $(-P^{*,\omega}x)\circ_{B}(-P^{*,\omega}y) = -\,P^{*,\omega}\bigl(x \circ_{B} (-P^{*,\omega}y)\bigr)$. Thus, $-P^{*,\omega}$ is an averaging operator on the descendent pre-Lie algebra $A_B$.
\end{proof}

\begin{theorem}\label{theorem:avequtobi}
Let $(A, \circ,B,\omega)$ be a quadratic Rota-Baxter pre-Lie algebra of weight $\lambda$ and $J_\omega:A^*\to A$ the induced linear isomorphism given by $\langle J_\omega^{-1}x,\,y\rangle:=\omega(x,y)$. Let $P:A\to A$ be an averaging operator on $(A, \circ,B,\omega)$. Then $(A, \circ, \Delta_{r^{B, \omega}}, P, -P)$ is an averaging pre-lie bialgebra.
\end{theorem}
\begin{proof}
By Theorem~\ref{theorem:frac-preliebi}, $(A, \circ, \Delta_{r^{B, \omega}})$ is a pre-Lie bialgebra. Since $P$ is an averaging operator on the pre-Lie algebra $(A, \circ)$, according to Definition ~\ref{def:apLbi}, we only need to show that $\bigl((A, \Delta_{r^{B, \omega}}),-P\bigr)$ is an averaging pre-Lie coalgebra. In other words, we need to show that $\bigl((A^*, \Delta_{r^{B, \omega}}^*:=\cdot_r),-P^*\bigr)$ is an averaging pre-Lie algebra.
Obviously,
\begin{equation}
P^{*}\circ J_{\omega}^{-1}=J_{\omega}^{-1}\circ P^{*,\omega}. \mlabel{eq:commutate}
\end{equation}
For any $f=J_{\omega}^{-1}(x)$, $g=J_{\omega}^{-1}(y)\in A^{*}$, we have,
\begin{align*}
(-P^*)f \cdot_r (-P^*)g
&= -P^*(J_\omega^{-1}x)\cdot_r (-P^*)(J_\omega^{-1}y) \\
&\overset{\eqref{eq:commutate}}{=} J_\omega^{-1}(-P^{*,\omega}x)\cdot_r J_\omega^{-1}(-P^{*,\omega}y) \\
&\overset{\eqref{eq:rb}}{=} -\frac{1}{\lambda} J_\omega^{-1}\!\left(
(-P^{*,\omega}x)\circ_B(-P^{*,\omega}y)
\right)
\\
&= \frac{1}{\lambda} J_\omega^{-1}(-P^{*,\omega})\!\left(
(-P^{*,\omega}x)\circ_B y
\right) \\
&\overset{\eqref{eq:commutate}}{=} -\frac{1}{\lambda} P^* J_\omega^{-1}\!\left(
(-P^{*,\omega}x)\circ_B y
\right) \\
&\overset{\eqref{eq:rb}}{=} (-P^*)\!\left(
J_\omega^{-1}(-P^{*,\omega}x)\cdot_r J_\omega^{-1}y
\right) \\
&\overset{\eqref{eq:commutate}}{=} (-P^*)\!\left(
(-P^*)(J_\omega^{-1}x)\cdot_r J_\omega^{-1}y
\right) \\
&= (-P^*)\bigl(
(-P^*)f \cdot_r g
\bigr).
\end{align*}
Similarly, $(-P^*)f \cdot_r (-P^*)g = (-P^*)\bigl(f \cdot_r (-P^*)g \bigr)$. Thus, $(A, \circ, \Delta_{r^{B, \omega}}, P, -P)$ is an averaging pre-lie bialgebra.
\end{proof}

\section{Admissible classical Yang-Baxter equation and relative Rota-Baxter operators} \label{sec:adalybinapl}

In this section, we introduce the notion of the classical Yang-Baxter equation in
an averaging pre-Lie algebra with respect to a linear map, whose solution gives rise to an averaging pre-Lie bialgebra. Then
we introduce the notion of relative Rota-Baxter operators on an averaging Lie algebra
with respect to a representation, which can give rise to solutions of the classical Yang-Baxter equation in the semidirect product averaging pre-Lie algebra.

\subsection{Admissible classical Yang-Baxter equation in the averaging pre-Lie algebra}

Suppose that $\bigl((A, \circ), P\bigr)$ is an $S$-admissible averaging pre-Lie algebra. In order to let $(A, \circ, \Delta_r, P, S)$ be an averaging pre-Lie bialgebra, we only need to further require that $\bigl((A^*, \Delta_r^*), S^*)$ is $P^*$-admissible averaging pre-Lie algebra, that is, $\bigl((A, \Delta_r), S)$ is an averaging pre-Lie coalgebra and ~\eqref{eq:aplb1} and ~\eqref{eq:aplb2} hold.

\begin{lemma}\label{lemma:averpreliebi}
Let $\bigl((A,\circ),P\bigr)$ be an $S$-admissible averaging pre-Lie algebra and $r = \sum_{i=1}^{n} a_i \otimes b_i \in A \otimes A$. Define a linear map $\Delta_r : A \to A \otimes A$ by Eq.~\eqref{eq:deltar}. Suppose that $\Delta_r^*$ defines a pre-Lie algebra structure on $A^*$. Then the following conclusions hold.
\begin{enumerate}
\item Eq.~\eqref{eq:aapLieco} holds if and only if, for all $x \in A$,
\begin{equation}
\begin{aligned}
\bigl(S\circ L_x \otimes S + S \otimes S\circ (L_x - R_x)\bigr)(r)
= \bigl(L_{S(x)} \circ P \otimes \mathrm{id} + S \otimes (L_{S(x)} - R_{S(x)})\bigr)(r) \\
(L_{S(x)} \otimes \mathrm{id})
  \bigl(P \otimes \mathrm{id} - \mathrm{id}\otimes S\bigr)(r)
= \bigl(\mathrm{id}\otimes L_{S(x)} - \mathrm{id}\otimes R_{S(x)}\bigr)
    \bigl(\mathrm{id}\otimes P  - S\otimes \mathrm{id}\bigr)(r)
\end{aligned} \mlabel{eq:combined1}
\end{equation}

\item Eq.~\eqref{eq:aplb1} holds if and only if, for all $x \in A$,
\begin{equation}
\begin{aligned}
&(P\circ L_x \otimes S + P\otimes S\circ (L_x - R_x))(r)
= \bigl(L_{P(x)}\circ P \otimes \mathrm{id}
   + P\otimes (L_{P(x)}-R_{P(x)})\bigr)(r) \\
&\bigl(L_{P(x)}\otimes \mathrm{id}
      + \mathrm{id}\otimes L_{P(x)}
      - \mathrm{id}\otimes R_{P(x)}\bigr)
 \bigl(\mathrm{id}\otimes S - P\otimes \mathrm{id}\bigr)(r)=0
\end{aligned} \mlabel{eq:combined2}
\end{equation}

\item Eq.~\eqref{eq:aplb2} holds if and only if, for all $x \in A$,
\begin{equation}
\begin{aligned}
&(S\circ L_x \otimes P + S\otimes P\circ (L_x - R_x))(r)
= \bigl(L_{P(x)}\circ S \otimes \mathrm{id}
   + S\otimes (L_{P(x)}-R_{P(x)})\bigr)(r) \\
&\bigl(
L_{P(x)}\otimes \mathrm{id}
+ \mathrm{id}\otimes (L_{P(x)}-R_{P(x)})
\bigr)
\bigl(S\otimes \mathrm{id}-\mathrm{id}\otimes P\bigr)(r)=0
\end{aligned} \mlabel{eq:combined3}
\end{equation}
\end{enumerate}
\end{lemma}
\begin{proof}
\begin{enumerate}
\item  For all $x \in A$, we have
\begin{align*}
(S \otimes S)\Delta_r(x)
&\overset{\eqref{eq:deltar}}{=} \sum_{i=1}^{n} (S\otimes S)\bigl(x\circ a_i \otimes b_i + a_i \otimes x\circ b_i - a_i \otimes b_i\circ x\bigr)\\
&= \sum_{i=1}^{n} \bigl(S(x\circ a_i)\otimes S(b_i) + S(a_i)\otimes S(x\circ b_i) - S(a_i)\otimes S(b_i\circ x)\bigr)\\
&= (S\circ L_x \otimes \mathrm{id})(\mathrm{id}\otimes S)(r) + (\mathrm{id}\otimes S\circ L_x)(S\otimes \mathrm{id})(r) - (\mathrm{id}\otimes S\circ R_x)(S\otimes \mathrm{id})(r)\\
(S\otimes \mathrm{id})\,\Delta_r(S(x))
&\overset{\eqref{eq:deltar}}{=}\sum_{i=1}^{n} (S\otimes \mathrm{id})\bigl(S(x)\circ a_i \otimes b_i + a_i \otimes S(x)\circ b_i - a_i \otimes b_i \circ S(x)\bigr)\\
&= \sum_{i=1}^{n} \bigl(S(S(x)\circ a_i)\otimes b_i + S(a_i)\otimes S(x)\circ b_i - S(a_i)\otimes b_i\circ S(x)\bigr)\\
&\overset{\eqref{eq:regu2}}{=} \sum_{i=1}^{n} \bigl(S(x)\circ P(a_i)\otimes b_i + S(a_i)\otimes S(x)\circ b_i - S(a_i)\otimes b_i\circ S(x)\bigr)\\
&= (L_{S(x)}\otimes \mathrm{id})(P\otimes \mathrm{id})(r)
   + (\mathrm{id}\otimes L_{S(x)})(S\otimes \mathrm{id})(r)
   - (\mathrm{id}\otimes R_{S(x)})(S\otimes \mathrm{id})(r)\\
(id\otimes S)\,\Delta_r(S(x))
&\overset{\eqref{eq:deltar}}{=} \sum_{i=1}^{n} (\mathrm{id}\otimes S)\bigl(S(x)\circ a_i \otimes b_i + a_i \otimes S(x)\circ b_i - a_i \otimes b_i \circ S(x)\bigr)\\
&= \sum_{i=1}^{n} \bigl(S(x)\circ a_i \otimes S(b_i) + a_i \otimes S(S(x)\circ b_i) - a_i \otimes S(b_i\circ S(x))\bigr)\\
&\overset{\eqref{eq:regu1} \eqref{eq:regu2}}{=} \sum_{i=1}^{n} \bigl(S(x)\circ a_i \otimes S(b_i) + a_i \otimes S(x)\circ P(b_i) - a_i \otimes P(b_i)\circ S(x)\bigr)\\
&= (L_{S(x)}\otimes \mathrm{id})(id\otimes S)(r)
   + (id\otimes L_{S(x)})(\mathrm{id}\otimes P)(r)
   - (id\otimes R_{S(x)})(\mathrm{id}\otimes P)(r)
\end{align*}
Then Eq.~\eqref{eq:aapLieco} holds if and only if Eq.~\eqref{eq:combined1} holds.

\item  For all $x \in A$, we obtain
\begin{align*}
(P\otimes S)\,\Delta_r(x)
&\overset{\eqref{eq:deltar}}{=} \sum_{i=1}^{n} (P\otimes S)\bigl(x\circ a_i \otimes b_i + a_i \otimes x\circ b_i - a_i \otimes b_i\circ x\bigr)\\
&= \sum_{i=1}^{n} \bigl( P(x\circ a_i)\otimes S(b_i) + P(a_i)\otimes S(x\circ b_i) - P(a_i)\otimes S(b_i\circ x)\bigr)\\
&= (P\circ L_x \otimes S)(r) + (P\otimes S\circ L_x)(r) - (P\otimes S\circ R_x)(r)\\
&= (P\circ L_x \otimes \mathrm{id})(\mathrm{id}\otimes S)(r)
   + (\mathrm{id}\otimes S\circ L_x)(P\otimes \mathrm{id})(r)
   - (\mathrm{id}\otimes S\circ R_x)(P\otimes \mathrm{id})(r)\\
(\mathrm{id}\otimes S)\,\Delta_r(P(x))
&\overset{\eqref{eq:deltar}}{=} \sum_{i=1}^{n} (\mathrm{id}\otimes S)\bigl(P(x)\circ a_i \otimes b_i + a_i \otimes P(x)\circ b_i - a_i \otimes b_i\circ P(x)\bigr)\\
&= \sum_{i=1}^{n} \bigl( P(x)\circ a_i \otimes S(b_i) + a_i \otimes S(P(x)\circ b_i) - a_i \otimes S(b_i\circ P(x))\bigr)\\
&\overset{\eqref{eq:regu1} \eqref{eq:regu2}}{=} \sum_{i=1}^{n} \bigl( P(x)\circ a_i \otimes S(b_i) + a_i \otimes P(x)\circ S(b_i) - a_i \otimes S(b_i)\circ P(x) \bigr)\\
&= (L_{P(x)}\otimes S)(r) + (\mathrm{id}\otimes L_{P(x)}\circ S)(r) - (\mathrm{id}\otimes R_{P(x)}\circ S)(r)\\
&= (L_{P(x)}\otimes \mathrm{id})(\mathrm{id}\otimes S)(r)
   + (\mathrm{id}\otimes L_{P(x)})(\mathrm{id}\otimes S)(r)
   - (\mathrm{id}\otimes R_{P(x)})(\mathrm{id}\otimes S)(r)\\
(P\otimes \mathrm{id})\,\Delta_r(P(x))
&\overset{\eqref{eq:deltar}}{=} \sum_{i=1}^{n} (P\otimes \mathrm{id})\bigl(P(x)\circ a_i \otimes b_i + a_i \otimes P(x)\circ b_i - a_i \otimes b_i\circ P(x)\bigr)\\
&= \sum_{i=1}^{n} \bigl( P\bigl(P(x)\circ a_i\bigr)\otimes b_i + P(a_i)\otimes P(x)\circ b_i - P(a_i)\otimes b_i\circ P(x) \bigr) \\
&\overset{\eqref{eq:averop}}{=} \sum_{i=1}^{n} \bigl( P(x)\circ P(a_i)\otimes b_i + P(a_i)\otimes P(x)\circ b_i - P(a_i)\otimes b_i\circ P(x) \bigr)\\
&= (L_{P(x)}\circ P \otimes id)(r) + (P\otimes L_{P(x)})(r) - (P\otimes R_{P(x)})(r)\\
&= (L_{P(x)}\otimes \mathrm{id})(P\otimes \mathrm{id})(r)
   + (\mathrm{id}\otimes L_{p(x)})(P\otimes \mathrm{id})(r)
   - (\mathrm{id}\otimes R_{P(x)})(P\otimes \mathrm{id})(r)
\end{align*}
Then Eq.~\eqref{eq:aplb1} holds if and only if Eq.~\eqref{eq:combined2} holds.

\item  For all $x \in A$, we obtain
\begin{align*}
(S\otimes P)\,\Delta_r(x)
&\overset{\eqref{eq:deltar}}{=} \sum_{i=1}^n (S\otimes P)\bigl(x\circ a_i \otimes b_i + a_i \otimes x\circ b_i - a_i \otimes b_i\circ x\bigr)\\
&= \sum_{i=1}^n \bigl(S(x\circ a_i)\otimes P(b_i) + S(a_i)\otimes P(x\circ b_i) - S(a_i)\otimes P(b_i\circ x)\bigr)\\
&= (S\circ L_x \otimes \mathrm{id})(\mathrm{id}\otimes P)(r)
   + (\mathrm{id}\otimes P\circ L_x)(S\otimes \mathrm{id})(r)
   - (\mathrm{id}\otimes P\circ R_x)(S\otimes \mathrm{id})(r)\\
(S\otimes \mathrm{id})\,\Delta_r(P(x))
&\overset{\eqref{eq:deltar}}{=} \sum_{i=1}^n (S\otimes \mathrm{id})\bigl(P(x)\circ a_i \otimes b_i + a_i \otimes P(x)\circ b_i - a_i \otimes b_i\circ P(x)\bigr)\\
&= \sum_{i=1}^n \bigl(S(P(x)\circ a_i)\otimes b_i + S(a_i)\otimes P(x)\circ b_i - S(a_i)\otimes b_i\circ P(x)\bigr)\\
&\overset{\eqref{eq:regu1}}{=} \sum_{i=1}^n \bigl(P(x)\circ S(a_i)\otimes b_i + S(a_i)\otimes P(x)\circ b_i - S(a_i)\otimes b_i\circ P(x)\bigr)\\
&= (L_{P(x)}\otimes \mathrm{id})(S\otimes \mathrm{id})(r)
   + (\mathrm{id}\otimes L_{P(x)})(S\otimes \mathrm{id})(r)
   - (\mathrm{id}\otimes R_{P(x)})(S\otimes \mathrm{id})(r)\\
(\mathrm{id}\otimes P)\,\Delta_r(P(x))
&\overset{\eqref{eq:deltar}}{=} \sum_{i=1}^n (\mathrm{id}\otimes P)\bigl(P(x)\circ a_i \otimes b_i + a_i \otimes P(x)\circ b_i - a_i \otimes b_i\circ P(x)\bigr)\\
&= \sum_{i=1}^n \bigl(P(x)\circ a_i \otimes P(b_i) + a_i \otimes P(P(x)\circ b_i) - a_i \otimes P(b_i\circ P(x))\bigr)\\
&\overset{\eqref{eq:averop}}{=}\sum_{i=1}^n \bigl(P(x)\circ a_i \otimes P(b_i) + a_i \otimes P(x)\circ P(b_i) - a_i \otimes P(b_i)\circ P(x)\bigr)\\
&= (L_{P(x)}\otimes \mathrm{id})(\mathrm{id}\otimes P)(r)
   + (\mathrm{id}\otimes L_{P(x)})(\mathrm{id}\otimes P)(r)
   - (\mathrm{id}\otimes R_{P(x)})(\mathrm{id}\otimes P)(r)
\end{align*}
Then Eq.~\eqref{eq:aplb2} holds if and only if Eq.~\eqref{eq:combined3} holds.
\end{enumerate}
\end{proof}

\begin{theorem}
Let $\bigl((A,\circ),P\bigr)$ be an $S$-admissible averaging pre-Lie algebra and $r \in A \otimes A$. Define a linear map $\Delta_r : A \to A \otimes A$ by Eq.~\eqref{eq:deltar}. Then $(A, \circ, \Delta_r, P, S)$ is an averaging pre-Lie bialgebra if and only if Eqs.~\eqref{eq:symmetaic}-\eqref{eq:S-equation} and ~\eqref{eq:combined1}-\eqref{eq:combined3} hold.
\end{theorem}
\begin{proof}
By Definition ~\ref{def:apLbi}, $(A, \circ, \Delta_r, P, S)$ is an averaging pre-Lie bialgebra if and only if $(A, \circ, \Delta_r)$ is a pre-Lie bialgebra and $\bigl((A^*, \Delta_r^*), S^*)$ is $P^*$-admissible averaging pre-Lie algebra. Then, by Theorem ~\ref{theorem:preliebi}, $(A, \circ, \Delta_r)$ is a pre-Lie bialgebra if and only if Eqs.~\eqref{eq:symmetaic}-\eqref{eq:S-equation} hold and by Lemma ~\ref{lemma:averpreliebi}, $\bigl((A^*, \Delta_r^*), S^*)$ is $P^*$-admissible averaging pre-Lie algebra if and only if ~\eqref{eq:combined1}-\eqref{eq:combined3} hold.
\end{proof}

In particular, by further simplifying the equivalent conditions, we obtain the following theorem.

\begin{theorem} \label{theorem:sapplbi}
Let $\bigl((A,\circ),P\bigr)$ be an $S$-admissible averaging pre-Lie algebra and $r \in A \otimes A$. Define a linear map $\Delta_r : A \to A \otimes A$ by Eq.~\eqref{eq:deltar}. Then $(A, \circ, \Delta_r, P, S)$ is an averaging pre-Lie bialgebra if Eq.~\eqref{eq:symmetaic} and the following equations hold:
\begin{align}
\sum_{i,j=1}^n
\bigl(
a_i \otimes b_i \circ a_j \otimes b_j
+ a_i \otimes a_j \otimes b_i \circ b_j
\bigr)
&=
\sum_{i,j=1}^n
\bigl(
a_i \circ a_j \otimes b_i \otimes b_j
+ a_i \otimes a_j \otimes b_j \circ b_i
\bigr), \mlabel{eq:cond0}\\
(S\otimes \mathrm{id}-\mathrm{id}\otimes P)(r) &= 0, \mlabel{eq:cond1}\\
(p\otimes \mathrm{id}-\mathrm{id}\otimes S)(r) &= 0. \mlabel{eq:cond2}
\end{align}
\end{theorem}
\begin{proof}
For all $x\in A$, we have
\begin{align*}
&\qquad \bigl( S\circ L_x \otimes S + S \otimes S\circ (L_x-R_x) \bigr)(r)
-\bigl( L_{S(x)}\circ P \otimes \mathrm{id} + S\otimes (L_{S(x)}-R_{S(x)}) \bigr)(r)\\
&= \bigl(S\circ L_x \otimes \mathrm{id}\bigr)\bigl(\mathrm{id}\otimes S\bigr)(r)
+ \bigl(\mathrm{id}\otimes S\circ(L_x-R_x)\bigr)\bigl(S\otimes \mathrm{id}\bigr)(r)
-\bigl(L_{S(x)}\circ P \otimes \mathrm{id}\bigr)(r)\\[4pt]
&\;
-\bigl(\mathrm{id}\otimes (L_{S(x)}-R_{S(x)})\bigr)\bigl(S\otimes \mathrm{id}\bigr)(r)\\
&\overset{\eqref{eq:cond1}\eqref{eq:cond2}}{=} \bigl(S\circ L_x \otimes \mathrm{id}\bigr)\bigl(P\otimes \mathrm{id}\bigr)(r)
-\bigl(L_{S(x)}\circ P \otimes \mathrm{id}\bigr)(r)
+ \bigl(\mathrm{id}\otimes S\circ(L_x-R_x)\bigr)\bigl(\mathrm{id}\otimes P\bigr)(r)\\[4pt]
&\;
-\bigl(\mathrm{id}\otimes (L_{S(x)}-R_{S(x)})\bigr)\bigl(\mathrm{id}\otimes P\bigr)(r)\\
&= \bigl(\bigl(S\circ L_x\circ P - L_{S(x)}\circ P\bigr)\otimes \mathrm{id}\bigr)(r)
+ \bigl(\mathrm{id}\otimes \bigl(S\circ(L_x-R_x)\circ P - (L_{S(x)}-R_{S(x)})\circ P\bigr)\bigr)(r)\\
&\overset{\eqref{eq:regu1} \eqref{eq:regu2}}{=} 0,
\end{align*}
which implies Eq.~\eqref{eq:combined1} holds.
Similarly, Eq.~\eqref{eq:combined2} and  Eq.~\eqref{eq:combined3} hold. Therefore, $(A, \circ, \Delta_r, P, S)$ is an averaging pre-Lie bialgebra if Eq.~\eqref{eq:symmetaic} and ~\eqref{eq:cond0}-~\eqref{eq:cond2} hodls.
\end{proof}
\noindent{Eq.~\eqref{eq:cond0} is the well-known \textbf{classical Yang-Baxter equation in $(A, \circ)$ or the $\mathbb{S}$-equation in $(A, \circ)$}.}

\begin{defn}
Let $\bigl((A,\circ),P\bigr)$ be an averaging pre-Lie algebra. Suppose that $r \in A\otimes A$ and $S:A \to A$ is a linear map. Then the $\mathbb{S}$-equation (i.e.Eq.~\eqref{eq:cond0}) together with Eq.~\eqref{eq:cond1} and Eq.~\eqref{eq:cond2} is called the \textbf{$S$-averaging $\mathbb{S}$-equation} or the \textbf{$S$-admissible classical Yang-Baxter equation in the averaging pre-Lie algebra $((A,\circ),P)$}.
\end{defn}

\begin{remark} \label{remark:rsym}
If $r$ is symmetric, then $\sum_{i=1}^{n} \bigl(S(a_i)\otimes b_i\bigr)
=
\sum_{i=1}^{n} \bigl(a_i \otimes P(b_i)\bigr)
$
is equivalent to
$
\sum_{i=1}^{n} \bigl(P(a_i)\otimes b_i\bigr)
=
\sum_{i=1}^{n} \bigl(a_i \otimes S(b_i)\bigr).
$
Therefore, Eq.~\eqref{eq:cond1} holds if and only if Eq.~\eqref{eq:cond2} holds.
\end{remark}

By Theorem~\ref{theorem:sapplbi} and Remark~\ref{remark:rsym}, we obtain

\begin{theorem} \label{theorem:ssapplbi}
Let $\bigl((A,\circ),P\bigr)$ be an $S$-admissible averaging pre-Lie algebra and $\Delta_r$ given by Eq.~\eqref{eq:deltar}.
If $r \in A\otimes A$ is a symmetric solution of the $S$-admissible classical Yang-Baxter equation in $((A,\circ),P)$. Then $(A,\circ,\Delta_r,P,S\bigr)$ is an averaging pre-Lie bialgebra.
\end{theorem}

\begin{proof}
By Remark~\ref{remark:rsym} and Eq.~\eqref{eq:cond1}, we get
\begin{align*}
(P\otimes \mathrm{id})r = (\mathrm{id}\otimes S)r.
\end{align*}

Therefore $(A,\circ,\Delta_r,P,S\bigr)$ is an averaging pre-Lie bialgebra by Theorem~\ref{theorem:sapplbi}.
\end{proof}

\subsection{Relative Rota-Baxter operators on averaging pre-Lie algebras}

We introduce the notion of a relative Rota-Baxter operator on an averaging pre-Lie algebra.

\begin{defn}
Let $\bigl((V,\rho,\varphi),\alpha\bigr)$ be a representation of an averaging pre-Lie algebra $((A,\circ),P)$. A linear map $T:V\to A$ is called a \textbf{relative Rota-Baxter operator} on $((A,\circ),P)$ with respect to the representation $\bigl((V,\rho,\varphi),\alpha\bigr)$ if the following equations hold:
\begin{align}
T(u)\circ T(v) &= T\bigl(\rho(T(u))v+\varphi(T(v))u\bigr),
,
\qquad
\forall u,v\in V \mlabel{eq:rerb1}\\
P\circ T &= T\circ \alpha. \mlabel{eq:rerb2}
\end{align}
Moreover, a relative Rota--Baxter operator on an averaging pre-Lie algebra with respect to the regular representation $\bigl((A,L,R),P\bigr)$
is called a \textbf{Rota--Baxter operator of weight $0$}.
\end{defn}

\begin{prop}\label{prop:desaverpl}
Let $T:V\to A$ be a relative Rota--Baxter operator on an averaging pre-Lie algebra $((A,\circ),P)$ with respect to the representation $\bigl((V,\rho,\varphi),\alpha\bigr)$.
Define a linear operation $\circ_T$ on $V$ by
\begin{equation}
u \circ_T v:=\rho(Tu)v + \varphi(Tv)u,
\qquad \forall\,u,v\in V.  \mlabel{eq:desprelie}
\end{equation}
Then $((V,\circ_T),\alpha)$ is an averaging pre-Lie algebra, which is called the \textbf{descendent averaging pre-Lie algebra}.
\end{prop}

\begin{proof}
For $ \forall u,v,w\in V$ and set $x:=Tu$, $y:=Tv$, $z:=Tw\in A$.
Since $T$ is a relative Rota--Baxter operator, we have $T(u\circ_T v)\overset{\eqref{eq:rerb1}}{=} Tu\circ Tv=x\circ y$.

On the one hand, we get
\begin{align*}
&\quad (u\circ_{T} v)\circ_{T} w-u\circ_{T}(v\circ_{T} w)\\
&= \rho \bigl(T(u\circ_{T} v)\bigr)w+\varphi(Tw)(u\circ_{T} v)
      -\rho(Tu)(v\circ_{T} w)-\varphi \bigl(T(v\circ_{T} w)\bigr)u \\
&= \rho(x\circ y)w+\varphi(z)\bigl(\rho(x)v+\varphi(y)u\bigr)
      -\rho(x)\bigl(\rho(y)w+\varphi(z)v\bigr)-\varphi(y\circ z)u \\
&= \rho(x\circ y)w+\varphi(z)\rho(x)v+\varphi(z)\varphi(y)u
      -\rho(x)\rho(y)w-\rho(x)\varphi(z)v-\varphi(y\circ z)u \\
&\overset{\eqref{eq:ropl1}\eqref{eq:ropl2}}{=} \rho(y\circ x)w-\rho(y)\rho(x)w
      +\varphi(z)\varphi(x)v-\varphi(x\circ z)v
      +\varphi(z)\rho(y)u-\rho(y)\varphi(z)u \\
&= (v\circ_{T} u)\circ_{T} w-v\circ_{T}(u\circ_{T} w),
\end{align*}
which implies that $(V,\circ_T)$ is a pre-Lie algebra.

On the other hand, we get
\begin{align*}
\alpha(u)\circ_{T}\alpha(v)
&= \rho\bigl(T(\alpha(u))\bigr)\alpha(v)
   +\varphi\bigl(T(\alpha(v))\bigr)\alpha(u) \\
&\overset{\eqref{eq:rerb2}}{=} \rho\bigl(P(Tu)\bigr)\alpha(v)
   +\varphi\bigl(P(Tv)\bigr)\alpha(u) \\
&\overset{\eqref{eq:rapL1}\eqref{eq:rapL2}}{=} \alpha\bigl(\rho(P(Tu))v\bigr)
   +\alpha\bigl(\varphi(Tv)\alpha(u)\bigr) \\
&\overset{\eqref{eq:rerb2}}{=} \alpha\bigl(\rho(T(\alpha(u)))v
   +\varphi(Tv)\alpha(u)\bigr) \\
&= \alpha\bigl(\alpha(u)\circ_{T} v\bigr),
\end{align*}
and similarly $\alpha(u)\circ_{T}\alpha(v) = \alpha\bigl(u\circ_{T} \alpha(v)\bigr)$,
which implies that $\alpha: V \to V$ is an averaging operator on the pre-Lie algebra $(V,\circ_T)$.

Therefore $((V,\circ_T), \alpha)$ is an averaging pre-Lie algebra.
\end{proof}

\begin{coro}
Let $T:V\to A$ be a relative Rota--Baxter operator on an averaging pre-Lie algebra $((A,\circ),P)$ with respect to the representation $\bigl((V,\rho,\varphi),\alpha\bigr)$.
Then $T$ is a homomorphism from the descendent averaging pre-Lie algebra
$((V,\circ_T), \alpha)$ to the averaging pre-Lie algebra $((A,\circ),P)$.
\end{coro}

\begin{prop} \label{prop:mmoaverprelie}
Let $T:V\to A$ be a relative Rota--Baxter operator on an averaging pre-Lie algebra $((A,\circ),P)$ with respect to the representation $\bigl((V,\rho,\varphi),\alpha\bigr)$. Then
\[
\Bigl(
\bigl((A,\circ),P\bigr),\
\bigl((V,\circ_T),\alpha\bigr),\
\rho,\varphi,\rho',\varphi'
\Bigr)
\]
is a matched pair of averaging pre-Lie algebras, where $((V,\circ_T), \alpha)$ is the descendent averaging pre-Lie algebra, the $\circ_T$ is given by Eq.~\eqref{eq:desprelie} and $\rho',\varphi': V \to \mathrm{End}(A)$ are given by
\begin{align*}
\rho'(u)x = - T(\varphi(x)u) + Tu \circ x,\qquad
\varphi'(u)x = - T(\rho(x)u) + x \circ Tu,\quad \forall u \in V, x\in A.
\end{align*}
\end{prop}

\begin{proof}
Based on Definition~\ref{def:mpoaverprelie}, the proof is divided into the following four steps.

\textbf{Step1:}
$\big((A,\circ),P\big)$ and $\big(V,\circ_T),\alpha \big)$ are averaging pre-Lie algebras.
It follows directly from the Proposition~\ref{prop:desaverpl}.

\textbf{Step2:}
$(\big(A,\rho',\varphi'),P\big)$ is a representation of $\big(V,\circ_T),\alpha \big)$. In other words, we need to prove Eqs.~\eqref{eq:ropl1}-~\eqref{eq:rapL2}.
First, for Eq.~\eqref{eq:ropl1}, we have
\begin{align*}
&\qquad \rho'\bigl(u \circ_T v - v \circ_T u\bigr)x
 - \rho'(u)\rho'(v)x + \rho'(v)\rho'(u)x \\[4pt]
&\overset{\eqref{eq:desprelie}}{=}\;
\rho'\Bigl(
 \rho(Tu)v + \varphi(Tv)u - \rho(Tv)u - \varphi(Tu)v
\Bigr)(x)
+ \rho'(u)\bigl(T(\varphi(x)v) - Tv \circ x\bigr)
- \rho'(v)\bigl(T(\varphi(x)u) - Tu \circ x\bigr) \\[4pt]
&=\;
T\Bigl(
 - \varphi(x)\rho(Tu)v
 - \varphi(x)\varphi(Tv)u
 + \varphi(x)\rho(Tv)u
 + \varphi(x)\varphi(Tu)v
\Bigr)
+ T(u \circ_T v - v \circ_T u)\circ x \\[4pt]
&\quad
- T\bigl(\varphi(T(\varphi(x)v))u\bigr)
+ Tu \circ T(\varphi(x)v)
+ T\bigl(\varphi(Tv \circ x)u\bigr)
- Tu \circ (Tv \circ x) \\[4pt]
&\quad
+ T\bigl(\varphi(T(\varphi(x)u))v\bigr)
- Tv \circ T(\varphi(x)u)
- T\bigl(\varphi(Tu\circ x)v\bigr)
+ Tv \circ (Tu\circ x) \\[6pt]
&\overset{\eqref{eq:desprelie}}{=}\;
T\Bigl(
 \underline{\color{red}{- \varphi(x)\rho(Tu)v}}
 - \underline{\color{blue}{\varphi(x)\varphi(Tv)u}}
 + \underline{\color{blue}{\varphi(x)\rho(Tv)u}}
 + \underline{\color{red}{\varphi(x)\varphi(Tu)v}}
\Bigr)
+ (Tu \circ Tv) \circ x
- (Tv \circ Tu) \circ x \\[4pt]
&\quad
+ \underline{\color{red}{T\bigl(\rho(Tu)\varphi(x)v\bigr)}}
+ \underline{\color{blue}{T\bigl(\varphi(Tv\circ x)u\bigr)}}
- Tu\circ (Tv\circ x) \\[4pt]
&\quad
- \underline{\color{blue}{T\bigl(\rho(Tv)\varphi(x)u\bigr)}}
- \underline{\color{red}{T\bigl(\varphi(Tu\circ x)v\bigr)}}
+ Tv\circ (Tu\circ x) \\[6pt]
&\overset{\eqref{eq:preLie} \eqref{eq:ropl2}}{=} \; 0.
\end{align*}
Next, for Eq.~\eqref{eq:ropl2}, we have
\begin{align*}
&\qquad \varphi'(u\circ_T v)x
 - \rho'(u)\bigl(\varphi'(v)x\bigr)
 + \varphi'(v)\bigl(\rho'(u)x\bigr)
 - \varphi'(v)\bigl(\varphi'(u)x\bigr) \\[4pt]
&\overset{\eqref{eq:desprelie}}{=}\;
- T\bigl(\rho(x)(u\circ_T v)\bigr)
+ x \circ T(u\circ_T v)
- \rho'(u)\bigl(-T(\rho(x)v)+x \circ Tv\bigr)
+ \varphi'(v)\bigl(-T(\varphi(x)u)+Tu \circ x\bigr)\\
&\quad
- \varphi'(v)\bigl(-T(\rho(x)u)+x \circ Tu\bigr) \\[6pt]
&\overset{\eqref{eq:desprelie}}{=}\;
- T\bigl(\rho(x)\rho(Tu)v + \rho(x)\varphi(Tv)u\bigr)
+ x\circ(Tu \circ Tv) \\
&\quad
- T\bigl(\varphi(T(\rho(x)v))u\bigr)
+ Tu\circ T(\rho(x)v)
+ T\bigl(\varphi(x\circ Tv)u\bigr)
- Tu\circ(x\circ Tv)  \\
&\quad
+ T\bigl(\rho(T(\varphi(x)u))v\bigr)
- T(\varphi(x)u)\circ Tv
- T\bigl(\rho(Tu \circ x)v\bigr)
+ (Tu\circ x)\circ Tv \\
&\quad
- T\bigl(\rho(T(\rho(x)u))v\bigr)
+ T(\rho(x)u)\circ Tv
+ T\bigl(\rho(x \circ Tu)v\bigr)
- (x\circ Tu)\circ Tv \\[6pt]
&=\;
- T\bigl(\rho(x)\rho(Tu)v + \rho(x)\varphi(Tv)u\bigr)
+ T\bigl(\rho(Tu)\rho(x)v\bigr)
+ T\bigl(\varphi(x\circ Tv)u\bigr)
- T\bigl(\varphi(Tv)\varphi(x)u\bigr) \\
&\quad
- T\bigl(\rho(Tu\circ x)v\bigr)
+ T\bigl(\varphi(Tv)\rho(x)u\bigr)
+ T\bigl(\rho(x\circ Tu)v\bigr) \\[6pt]
&\overset{\eqref{eq:preLie} \eqref{eq:ropl2}}{=}\; 0 .
\end{align*}
Finally, for Eq.~\eqref{eq:rapL1} and Eq.~\eqref{eq:rapL2}, we only prove $\rho'(\alpha(u))P(x) =  P\bigl(\rho'(\alpha(u))\,x\bigr)$ and the others are similar. Indeed,
\begin{align*}
\rho'(\alpha(u))P(x)
&= -T\bigl(\varphi(P(x))\alpha(u)\bigr)
   + T(\alpha(u))\circ P(x) \\[4pt]
&\overset{\eqref{eq:rapL2}}{=} -T\bigl(\alpha(\varphi(x)\alpha(u))\bigr)
   + P(Tu)\circ P(x) \\[4pt]
&\overset{\eqref{eq:averop}}{=} -P\bigl(T(\varphi(x)\alpha(u))\bigr)
   + P\bigl(P(Tu)\circ x\bigr) \\[4pt]
&= P\bigl(-T(\varphi(x)\alpha(u)) + T(\alpha(u))\circ x\bigr) \\[4pt]
&= P\bigl(\rho'(\alpha(u))\,x\bigr).
\end{align*}

\textbf{Step3:}
$\bigl((V,\rho,\varphi),\alpha\bigr)$ is a representation of $\big((A,\circ),P\big)$.
This is exactly the hypothesis that $\bigl((V,\rho,\varphi),\alpha\bigr)$
is a representation of $((A,\circ),P)$.

\textbf{Step4:}
$\big((A,\circ),(V,\circ_T),\rho,\varphi,\rho',\varphi'\big)$
is a matched pair of the pre-Lie algebras $(A,\circ)$ and $(V,\circ_T)$.

Substituting the above actions into
Eqs.~\eqref{eq:mppl1}--~\eqref{eq:mppl4} and using again
$T(u\circ_T v)=Tu\circ Tv$ and Eqs.~\eqref{eq:ropl1}--~\eqref{eq:ropl2},
we obtain that $\bigl((A,\circ),(V,\circ_T),\rho,\varphi,\rho',\varphi'\bigr)$
is a matched pair of pre-Lie algebras.

Therefore the sextuple
$\bigl(((A,\circ),P),((V,\circ_T),\alpha),\rho,\varphi,\rho',\varphi'\bigr)$
is a matched pair of averaging pre-Lie algebras.
\end{proof}

\begin{prop}\label{prop:rRBclYB}
Let $((A,\circ),P)$ be an averaging pre-Lie algebra. Suppose that $r \in A\otimes A$ is symmetric and $S: A \to A$ is a linear map.
If $r^{\sharp}\colon A^* \to A$ is a relative Rota--Baxter operator on $((A,\circ),P)$ with respect to the
representation $((A^*,L^* - R^*,-R^*), S^*)$, then $r$ is a solution of the $S$-admissible classical Yang-Baxter equation in the averaging pre-Lie algebra $((A,\circ),P)$.
\end{prop}
\begin{proof}
Since $r^{\sharp}\colon A^* \to A$ is a relative Rota-Baxter operator on $((A,\circ),P)$ with respect to the
representation $((A^*,L^* - R^*,-R^*), S^*)$, i.e.
\begin{equation}
r^{\sharp}(a^{*})\circ r^{\sharp}(b^{*})
=
r^{\sharp}\Bigl((L^{*}-R^{*})(r^{\sharp}(a^{*}))(b^{*})
- R^{*}(r^{\sharp}(b^{*}))(a^{*})\Bigr), \mlabel{eq:sharp1}
\end{equation}
and
\begin{equation}
P \circ r^{\sharp} = r^{\sharp} \circ S^*.  \mlabel{eq:sharp2}
\end{equation}
Then first by \cite[Theorem~6.6]{Bai08}, we know that $r$ is a solution of the $\mathbb{S}$-equation if and only if ~\eqref{eq:sharp1} holds.

Next, for all $x \in A$, $a^* \in A^*$, one gets
\begin{align*}
P\bigl(r^{\sharp}(a^{*})\bigr)
=
P\left(\sum_{i=1}^{n}\langle a^{*},a_i\rangle b_i\right)
=
\sum_{i=1}^{n}\langle a^{*},a_i\rangle P(b_i),\\
r^{\sharp}\bigl(S^{*}(a^{*})\bigr)
=
\sum_{i=1}^{n}\langle S^{*}(a^{*}),a_i\rangle b_i
=
\sum_{i=1}^{n}\langle a^{*},S(a_i)\rangle b_i.
\end{align*}
So Eq.~\eqref{eq:sharp2} holds if and only if Eq.~\eqref{eq:cond1} holds.
Since $r$ is symmetric, Eq.~\eqref{eq:cond2} also holds.
Then $r$ is a solution of the $S$-admissible classical Yang-Baxter equation in the averaging pre-Lie algebra $((A,\circ),P)$.
This completes the proof.
\end{proof}

In a quadratic averaging pre-Lie algebra, relative Rota-Baxter operators with respect to the coregular representation are equivalent to Rota-Baxter operators of weight 0.

\begin{prop}\label{prop:rRBandRB}
Let $\bigl((A,\circ), P, \omega\bigr)$ be a quadratic averaging pre-Lie algebra and $T \colon A^* \to A$ a linear map.
Then $T$ is a relative Rota--Baxter operator on $\bigl((A,\circ), P\bigr)$ with respect to the coregular representation
$\bigl( (A^*, L^*-R^*, -R^*), P^*)$ if and only if $T \circ \omega^{\sharp}$ is a Rota--Baxter operator of weight $0$ on the averaging pre-Lie algebra $\bigl((A,\circ), P\bigr)$.
\end{prop}

\begin{proof}
For all $x, y \in A$, by Theorem ~\ref{theorem:omegasharp}, we have
\begin{align*}
T\bigl(\omega^{\sharp}(x)\bigr)\circ T\bigl(\omega^{\sharp}(y)\bigr)
&=T\Bigl((L^{*}-R^{*})\bigl(T(\omega^{\sharp}x)\bigr)\bigl(\omega^{\sharp}y\bigr)
   -R^{*}\bigl(T(\omega^{\sharp}y)\bigr)\bigl(\omega^{\sharp}x\bigr)\Bigr)\\
&=T\Bigl(\omega^{\sharp}\bigl(L(T(\omega^{\sharp}x))\,y\bigr)+\omega^{\sharp}\bigl(R(T(\omega^{\sharp}y))\,x\bigr)\Bigr)\\
&=T\Bigl(\omega^{\sharp}\bigl(T(\omega^{\sharp}x)\circ y\bigr)+\omega^{\sharp}\bigl(x\circ T(\omega^{\sharp}y)\bigr)\Bigr),
\end{align*}
which implies that $T$ is a relative Rota--Baxter operator if and only if $T \circ \omega^{\sharp}$ is a Rota--Baxter operator of weight 0 on the pre-Lie algebra $(A, \circ)$.

On the other hand, $P \circ (T \circ \omega^{\sharp}) = (T \circ \omega^{\sharp}) \circ P$
if and only if $(P \circ T) \circ \omega^{\sharp} = (T \circ P^{*}) \circ \omega^{\sharp},$
which is equivalent to $P \circ T = T \circ P^{*}$ due to the fact that the bilinear form $\omega$ is nondegenerate.

Thus, $T$ is a relative Rota--Baxter operator if and only if $T \circ \omega^{\sharp}$ is a Rota--Baxter operator of weight $0$ on the
averaging pre-Lie algebra $\bigl((A,\circ), P\bigr)$.
\end{proof}

\begin{coro}
Let $\bigl((A,\circ), P, \omega\bigr)$ be a quadratic averaging pre-Lie algebra. Suppose that $r \in A \otimes A$ is symmetric.
If $r^{\sharp} \circ \omega^{\sharp}$ is a Rota--Baxter operator of weight $0$ on the averaging pre-Lie algebra $\bigl((A,\circ), P\bigr)$, then $r$ is a solution of the classical Yang-Baxter equation in the averaging pre-Lie algebra $\bigl((A,\circ), P\bigr)$.
\end{coro}

\begin{proof}
It follows directly from Proposition~\ref{prop:rRBclYB} and Proposition~\ref{prop:rRBandRB}.
\end{proof}

\begin{theorem}\label{theorem:equiva3}
Let $\bigl((A,\circ),P\bigr)$ be an averaging pre-Lie algebra, $(V,\rho,\varphi)$ a representation of $(A,\circ)$,
$S\colon A \to A$ and $\alpha,\beta \colon V \to V$ linear maps.
Then the following conditions are equivalent:
\begin{enumerate}
  \item There is an averaging pre-Lie algebra
  $
  \bigl(A \ltimes_{\rho,\varphi} V,\; P+\alpha\bigr)
  $
  such that the linear map $S+\beta$ on $A\oplus V$
  is admissible to $\bigl(A \ltimes_{\rho,\varphi} V,\; P+\alpha\bigr)$.  \mlabel{item:a}

  \item There is an averaging pre-Lie algebra
  $
  \bigl(A \ltimes_{\rho^*-\varphi^*,\, -\varphi^*} V^*,\; P+\beta^*\bigr)
  $
  such that the linear map $S+\alpha^*$ on $A\oplus V^*$
  is admissible to
  $
  \bigl(A \ltimes_{\rho^*-\varphi^*,\, -\varphi^*} V^*,\; N+\beta^*\bigr).
  $  \mlabel{item:b}

  \item The following conditions are satisfied:
  \begin{enumerate}
    \item $(V,\rho,\varphi,\alpha)$ is a representation of $\bigl((A,\circ),P\bigr)$; \mlabel{item:c1}
    \item $S$ is admissible to $\bigl((A,\circ),P\bigr)$;\mlabel{item:c2}
    \item $\beta$ is admissible to $\bigl((A,\circ),P\bigr)$ on $(V,\rho,\varphi)$;\mlabel{item:c3}
    \item For all $x\in A$ and $v\in V$, the following equations hold:
    \begin{align}
      \beta\bigl(\rho(x)\alpha(v)\bigr)
      =
      \rho\bigl(S(x)\bigr)\alpha(v)
      =
      \beta\bigl(\rho(S(x))v\bigr), \mlabel{eq:c41} \\
      \beta\bigl(\varphi(x)\alpha(v)\bigr)
      =
      \varphi\bigl(S(x)\bigr)\alpha(v)
      =
      \beta\bigl(\varphi(S(x))v\bigr). \mlabel{eq:c42}
    \end{align}
  \end{enumerate}  \mlabel{item:c}
\end{enumerate}
\end{theorem}
\begin{proof}
\textbf{Step1:$~\eqref{item:a} \iff ~\eqref{item:c}$.}

By Proposition ~\ref{prop:semi-repre}, $\bigl(A \ltimes_{\rho,\varphi} V,\; P+\alpha\bigr)$ is an averaging pre-Lie algebra if and only if $(V,\rho,\varphi,\alpha)$ is a representation of $\bigl((A,\circ),P\bigr)$. To prove the equivalence of the remaining parts, we need to carry out the following small computations.

For $x,y \in A$ and $u,v \in V$, we have
\begin{align*}
(P+\alpha)(x+u)\circ_{\heartsuit}(S+\beta)(y+v)
&=(P(x)+\alpha(u))\circ_{\heartsuit}\bigl(S(y)+\beta(v)\bigr)\\
&=P(x)\circ S(y)+\rho \bigl(P(x)\bigr)\beta(v)+\varphi \bigl(S(y)\bigr)\alpha(u),\\
(S+\beta)\Bigl((P+\alpha)(x+u)\circ_{\heartsuit}(y+v)\Bigr)
&=(S+\beta)\Bigl((P(x)+\alpha(u))\circ_{\heartsuit}(y+v)\Bigr)\\
&=(S+\beta)\Bigl(P(x)\circ y+\rho\bigl(P(x)\bigr)v+\varphi(y)\alpha(u)\Bigr)\\
&=S\bigl(P(x)\circ y\bigr)+\beta\Bigl(\rho\bigl(P(x)\bigr)v+\varphi(y)\alpha(u)\Bigr),\\
(S+\beta)\Bigl((x+u)\circ_{\heartsuit}(S+\beta)(y+v)\Bigr)
&=(S+\beta)\Bigl((x+u)\circ_{\heartsuit}\bigl(S(y)+\beta(v)\bigr)\Bigr)\\
&=(S+\beta)\Bigl(x\circ S(y)+\rho(x)\beta(v)+\varphi\bigl(S(y)\bigr)u\Bigr)\\
&=S\bigl(x\circ S(y)\bigr)+\beta\Bigl(\rho(x)\beta(v)+\varphi\bigl(S(y)\bigr)u\Bigr),
\end{align*}
\begin{align*}
(S+\beta)(x+u)\circ_{\heartsuit}(P+\alpha)(y+v)
&=\bigl(S(x)+\beta(u)\bigr)\circ_{\heartsuit}\bigl(P(y)+\alpha(v)\bigr)\\
&=S(x)\circ P(y)+\rho\bigl(S(x)\bigr)\alpha(v)+\varphi\bigl(P(y)\bigr)\beta(u),\\
(S+\beta)\Bigl((x+u)\circ_{\heartsuit}(P+\alpha)(y+v)\Bigr)
&=(S+\beta)\Bigl((x+u)\circ_{\heartsuit}\bigl(P(y)+\alpha(v)\bigr)\Bigr)\\
&=(S+\beta)\Bigl(x\circ P(y)+\rho(x)\alpha(v)+\varphi\bigl(P(y)\bigr)u\Bigr)\\
&=S\bigl(x\circ P(y)\bigr)+\beta\Bigl(\rho(x)\alpha(v)+\varphi\bigl(P(y)\bigr)u\Bigr),\\
(S+\beta)\Bigl((S+\beta)(x+u)\circ_{\heartsuit}(y+v)\Bigr)
&=(S+\beta)\Bigl(\bigl(S(x)+\beta(u)\bigr)\circ_{\heartsuit}(y+v)\Bigr)\\
&=(S+\beta)\Bigl(S(x)\circ y+\rho\bigl(S(x)\bigr)v+\varphi(y)\beta(u)\Bigr)\\
&=S\bigl(S(x)\circ y\bigr)+\beta\Bigl(\rho\bigl(S(x)\bigr)v+\varphi(y)\beta(u)\Bigr).
\end{align*}
Therefore, based on the equivalence correspondence shown in the Table~\ref{table:aandc} below, the equivalence of the remaining parts in this step follows.
\begin{table}[htbp]
\label{tab:shifted-equations}
\[
\begin{array}{|c|c|}
\hline
\textbf{Shifted equation}
&
\textbf{Equivalent system of equations} \\
\hline
\begin{array}{c}
\text{Eq.~\eqref{eq:regu1} holds}\\
\bigl(P\to P+\alpha,\ S\to S+\beta,\\
x\to x+u,\ y\to y+v\bigr)
\end{array}
&
\begin{array}{l}
\text{Eq.~\eqref{eq:regu1} holds},\\
\text{Eq.~\eqref{eq:rapL1} holds}(\alpha \to \beta),\\
\text{Eq.~\eqref{eq:c42} holds}(x \to y,\ v \to u).
\end{array}
\\
\hline
\begin{array}{c}
\text{Eq.~\eqref{eq:regu2} holds}\\
\bigl(P\to P+\alpha,\ S\to S+\beta,\\
x\to x+u,\ y\to y+v\bigr)
\end{array}
&
\begin{array}{l}
\text{Eq.~\eqref{eq:regu2} holds},\\
\text{Eq.~\eqref{eq:rapL2} holds}(\alpha \to \beta,\ x \to y,\ v \to u),\\
\text{Eq.~\eqref{eq:c41} holds}.
\end{array}
\\
\hline
\end{array}
\]
\centering
\caption{The equivalence of the remaining parts between Item ~\eqref{item:a} and Item ~\eqref{item:c}} \label{table:aandc}
\end{table}

Hence, Item~\eqref{item:a} $\iff$ Item~\eqref{item:c}.

\textbf{Step2:$~\eqref{item:b} \iff ~\eqref{item:c}$.}

Based on Item~\eqref{item:a} $\iff$ Item~\eqref{item:c}, we have ~\eqref{item:b} holds if and only if
\begin{enumerate}[label=\textcircled{\scriptsize\arabic*}]
  \item $(V^*,\rho^* - \varphi^*, -\varphi^*,\beta^*)$ is a representation of $\bigl((A,\circ),P\bigr)$; \mlabel{item:b1}

  \item $S$ is admissible to $\bigl((A,\circ),P\bigr)$; \mlabel{item:b2}

  \item $\alpha^*$ is admissible to $\bigl((A,\circ),P\bigr)$ on $(V^*,\rho^* - \varphi^*, -\varphi^*)$; \mlabel{item:b3}

  \item For all $x\in A$ and $v^*\in V^*$, the following equations hold:
  \begin{align}
    &\alpha^*\bigl(\rho^* - \varphi^*)(x)\beta^*(v^*)\bigr)
    =
    (\rho^* - \varphi^*)\bigl(S(x)\bigr)\beta^*(v^*)
    =
    \alpha^*\bigl((\rho^* - \varphi^*)(S(x))v^*\bigr), \mlabel{eq:231} \\
    &\alpha^*\bigl(-\varphi^*(x)\beta^*(v^*)\bigr)
    =
    -\varphi^*\bigl(S(x)\bigr)\beta^*(v^*)
    =
    \alpha^*\bigl(-\varphi^*(S(x))v^*\bigr) \mlabel{eq:232}.
  \end{align}
\end{enumerate}
Then, the equivalence between Item~\eqref{item:b} and Item~\eqref{item:c} follows from the correspondence shown in the Table~\ref{table:bandc} below.
\begin{table}[htbp]
\begin{tabular}{|c|c|}
\hline
\textbf{Left conditions} & \textbf{Equivalent right conditions} \\
\hline
\ref{item:b1} & \eqref{item:c3} \\
\hline
\ref{item:b2} & \eqref{item:c2} \\
\hline
\ref{item:b3} & \eqref{item:c1} \\
\hline
\eqref{eq:231} & \eqref{eq:c41}, \eqref{eq:c42} \\
\hline
\eqref{eq:232} & \eqref{eq:c42} \\
\hline
\end{tabular}
\centering
\caption{The equivalence of the remaining parts between Item ~\eqref{item:a} and Item ~\eqref{item:c}} \label{table:bandc}
\end{table}

This completes the proof.
\end{proof}

\begin{theorem}
Let $\bigl((A,\circ), P\bigr)$ be an $\beta$-admissible averaging pre-Lie algebra on $(V,\rho,\varphi)$. Let $S \colon A \to A$, $\alpha \colon V \to V$ be linear maps.
Let $T \colon V \to A$ be a linear map which is identified as an element in
$V^* \otimes A \subset (A \ltimes_{\rho^*-\varphi^*,-\varphi^*} V^*) \otimes
(A \ltimes_{\rho^*-\varphi^*,-\varphi^*} V^*)$.
\begin{enumerate}
\item
The element $r = T + \tau(T)$ is a symmetric solution of the $(S+\alpha^*)$-averaging $\mathbb{S}$-equation in the averaging pre-Lie algebra $(A \ltimes_{\rho^*-\varphi^*,-\varphi^*} V^*, P+\beta^*)$ if and only if $T$ is a relative Rota-Baxter operator on $\bigl((A, \circ), P\bigr)$ with the respect to the representation $\bigl((V,\rho,\varphi), \alpha\bigr)$, and $T\beta = S T$.  \mlabel{item:srb11}

\item
Assume that $\bigl((V,\rho,\varphi),\alpha\bigr)$ is a representation of $((A,\circ),P)$. If $T$ is a relative Rota-Baxter operator on $\bigl((A, \circ), P\bigr)$ with the respect to the representation $\bigl((V,\rho,\varphi), \alpha\bigr)$ and
$T \circ \beta = S \circ T$, then $r = T + \tau(T)$ is a symmetric solution of the $(S+\alpha^*)$-averaging $\mathbb{S}$-equation in the averaging pre-Lie algebra $(A \ltimes_{\rho^*-\varphi^*,-\varphi^*} V^*, P+\beta^*)$.
Moreover, if $(A,P)$ is $S$-admissible and Eqs.~\eqref{eq:c41} and~\eqref{eq:c42} hold, then
$(A \ltimes_{\rho^*-\varphi^*,-\varphi^*} V^*, P+\beta^*)$ is $(S+\alpha^*)$-admissible.
In this case, there is an averaging pre-Lie bialgebra $(A \ltimes_{\rho^*-\varphi^*,-\varphi^*} V^*, P+\beta^*, \delta_r, S+\alpha^*)$,
where $\delta_r$ is defined by Eq.~\eqref{eq:deltar} with $r = T + \tau(T)$. \mlabel{item:srb12}
\end{enumerate}
\end{theorem}

\begin{proof}
~\eqref{item:srb11}. Let $\{e_1,e_2,\ldots,e_n\}$ be a basis of $V$ and $\{e^1,e^2,\ldots,e^n\}$ be its dual basis. Then $r=T+\tau(T)$ corresponds to
\[
\sum_{i=1}^n \bigl(T(e_i)\otimes e^i + e^i \otimes T(e_i)\bigr)
\in
(A \ltimes_{\rho^*,-\varphi^*} V^*)
\otimes
(A \ltimes_{\rho^*,-\varphi^*} V^*).
\]
By \cite[Theorem~5.26]{GM}, we find that $r=T+\tau(T)$ is a symmetric solution of the $S$-equation in $A \ltimes_{\rho^*-\varphi^*,-\varphi^*} V^*$ if and only if Eq.~\eqref{eq:sharp1} holds for the regular representation $(A \ltimes_{\rho^*,-\varphi^*} V^*;\,L,R)$ of
$A \ltimes_{\rho^*-\varphi^*,-\varphi^*} V^*$.
Note that
\begin{align*}
\bigl( \mathrm{id}\otimes(P+\beta^*) \bigr)(r)
&=
\sum_{i=1}^n \Bigl( e^i \otimes P(T(e_i)) + T(e_i) \otimes \beta^*(e^i)\Bigr),\\
\bigl((S+\alpha^*\otimes \mathrm{id}\bigr)(r)
&=
\sum_{i=1}^n \Bigl( \alpha^*(e^i)\otimes T(e_i) + S(T(e_i)) \otimes e^i \Bigr).
\end{align*}
Further,
\begin{align*}
\sum_{i=1}^n \beta^*(e^i)\otimes T(e_i)
&=
\sum_{i=1}^n \sum_{j=1}^n \langle \beta^*(e^i), e_j\rangle\, e^j \otimes T(e_i)\\
&=
\sum_{j=1}^n e^j \otimes \sum_{i=1}^n \langle e^i, \beta(e_j)\rangle\, T(e_i)\\
&=
\sum_{i=1}^n e^i \otimes \sum_{j=1}^n \langle e^j, \beta(e_i)\rangle\, T(e_j)
=
\sum_{i=1}^n e^i \otimes T(\beta(e_i)),
\\[1ex]
\sum_{i=1}^n T(e_i)\otimes \alpha^*(e^i)
&=
\sum_{i=1}^n \sum_{j=1}^n T(e_i)\otimes \langle \alpha^*(e^i), e_j\rangle\, e^j\\
&=
\sum_{j=1}^n \Bigl( \sum_{i=1}^n T(e_i)\langle e^i, \alpha(e_j)\rangle \Bigr)\otimes e^j
=
\sum_{i=1}^n T(\alpha(e_i))\otimes e^i.
\end{align*}

Therefore,
\[
\bigl( \mathrm{id}\otimes(P+\beta^*) \bigr)(r)
=
\bigl((S+\alpha^*)\otimes \mathrm{id}\bigr)(r)
\]
if and only if
\[
T\beta = S T
\quad \text{and} \quad
P T = T\alpha.
\]
This completes the proof.

~\eqref{item:srb12}. It follows Item ~\eqref{item:srb11} and Theorem ~\ref{theorem:equiva3}.
\end{proof}

\section{Balanced averaging pre-Lie bialgebras} \label{sec:balancedaverpl}

In this section, we introduce a new notion of averaging Lie bialgebras and show that, under suitable conditions, an averaging Lie bialgebra can be obtained from an averaging pre-Lie bialgebra. We start by recalling the notions of Lie coalgebras, Lie bialgebras and averaging Lie algebras. More details can be found in \cite{HSZh26}.

\begin{defn}
A \textbf{Lie coalgebra} is a pair $(A,\delta)$, where
\begin{enumerate}
\item $A$ is a vector space;
\item $\delta: A \to A\otimes A$ is a linear map;
\item $\delta$ is \textbf{antisymmetric}, that is,
\begin{align*}
\delta = -\,\tau\circ\delta;
\end{align*}

\item $\delta$ satisfies the \textbf{co-Jacobi identity}:
\begin{align*}
(\mathrm{id}_A  + \sigma + \sigma^2)(\mathrm{id}_A \otimes\delta)\delta = 0,
\end{align*}
where $\sigma(x\otimes y\otimes z):=z\otimes x\otimes y$.
\end{enumerate}
\end{defn}

\begin{defn}
A \textbf{Lie bialgebra} is a triple $(A,[\cdot,\cdot],\Delta)$, where
\begin{enumerate}
\item $(A,[\cdot,\cdot])$ is a Lie algebra;
\item $(A,\delta)$ is a Lie coalgebra;
\item For all $x,y\in A$, the following compatibility condition holds:
\begin{align*}
\delta([x,y])
= (\mathrm{ad}_x\otimes \mathrm{id}_A + \mathrm{id}_A \otimes \mathrm{ad}_x)\delta(y)
- (\mathrm{ad}_y\otimes \mathrm{id}_A + \mathrm{id}_A \otimes \mathrm{ad}_y)\delta(x).
\end{align*}
\end{enumerate}
\end{defn}

\begin{defn}
Let $(A,[\cdot,\cdot])$ be a Lie algebra and $P : A \to A$ a linear map.
If
\begin{align*}
[P(x) , P(y)] = P\Big([P(x),y] \Big) = P\Big([x , P(y) ]\Big) ,
\qquad x,y \in A, \mlabel{eq:ala}
\end{align*}
then we call $P$ an \textbf{averaging operator} on $(A,[\cdot,\cdot])$ and the pair $\big( (A,[\cdot,\cdot]), P \big)$ an \textbf{averaging Lie algebra}.
\end{defn}

Next, we introduce the notions of averaging Lie coalgebras and averaging Lie bialgebras.

\begin{defn}
An \textbf{averaging Lie coalgebra} is a pair $\big( (A,\delta), S\big)$, where $(A,\delta)$ is a Lie coalgebra and $S : A \to A$ is a linear map such that
\begin{equation}
(S \otimes S)\delta(x) = (S \otimes \mathrm{id}_A)\delta(S(x))) = (\mathrm{id}_A \otimes S)\delta(S(x)))
\qquad x \in A.  \mlabel{eq:apLieco}
\end{equation}
\end{defn}

\begin{defn}\label{def:averagingLiebialgebra}
An \textbf{averaging Lie bialgebra} is a quintuple
\[
(A,[\cdot,\cdot],\delta,P,S),
\]
where
\begin{enumerate}
\item $\bigl((A,[\cdot,\cdot]),P\bigr)$ is an averaging Lie algebra;

\item $\bigl((A,\delta),S\bigr)$ is an averaging Lie coalgebra;

\item $(A,[\cdot,\cdot],\delta)$ is a Lie bialgebra;

\item For all $x,y\in A$, the following compatibility conditions hold:
\begin{equation}
[P(x),S(y)]
= S\bigl([P(x),y]\bigr)
= S\bigl([x,S(y)]\bigr), \mlabel{eq:aLbi1}
\end{equation}
\begin{equation}
(S\otimes P)\delta(x)
= (S\otimes \mathrm{id}_A)\delta(P(x))
= (\mathrm{id}_A \otimes P)\delta(P(x)). \mlabel{eq:aLbi2}
\end{equation}
\end{enumerate}
\end{defn}

\begin{remark}
When $S=P$, the notion of an averaging Lie bialgebra in the sense of Definition~\ref{def:averagingLiebialgebra}
coincides with the notion of an averaging Lie bialgebra introduced in \cite[Definition~4.4]{HSZh26}.
\end{remark}

\begin{defn}\cite{GM}
A pre-Lie bialgebra $(A,\circ,\Delta)$ is called \textbf{balanced} if
\[
x_{(1)}\circ y \otimes x_{(2)}
+ y_{(2)}\otimes y_{(1)}\circ x
=
y_{(1)}\circ x \otimes y_{(2)}
+ x_{(2)}\otimes x_{(1)}\circ y,
\qquad \forall\, x,y\in A.
\]
\end{defn}

\begin{theorem}\cite{GM}\label{Theorem:Liebi}
Let $(A,\circ,\Delta)$ be a pre-Lie bialgebra. For $\forall\, x,y\in A$, we define
\begin{equation}
[x,y]:=x\circ y-y\circ x, \mlabel{eq:bplbi1}
\end{equation}
\begin{equation}
\delta(x):=\Delta(x) - \tau(\Delta(x)). \mlabel{eq:bplbi2}
\end{equation}
Then $(A,[\cdot,\cdot],\delta)$ is a Lie bialgebra if and only if $(A,\circ,\Delta)$ is balanced.
\end{theorem}

Averaging Lie bialgebras can be obtained from averaging pre-Lie bialgebras in the following way.

\begin{theorem}\label{theorem:aplb-avlb}
Let $(A,\circ,\Delta,P,S)$ be an averaging pre-Lie bialgebra. Then $(A,[\cdot,\cdot],\delta,P,S)$, where $[\cdot,\cdot]$ and $\delta$ are given by Eqs.~\eqref{eq:bplbi1} and ~\eqref{eq:bplbi2}, is an averaging Lie bialgebra if and only if $(A,\circ,\Delta)$ is balanced.
\begin{proof}
We split the proof into four steps.

\textbf{Step 1.}
By \cite[Proposition~5.17]{HSZh26}, we know that $\bigl((A,[\cdot,\cdot]),P\bigr)$ is an averaging Lie algebra.

\textbf{Step 2.}
We prove that $\bigl((A,\delta),S\bigr)$ is an averaging Lie coalgebra.

First, $\delta$ is antisymmetric. Indeed,
\begin{align*}
\delta(x)
\overset{\eqref{eq:bplbi2}}{=} \Delta(x)-\tau\bigl(\Delta(x)\bigr)
\overset{\tau^2=\mathrm{id}_{A\otimes A}}{=} -\,\tau\Bigl(\Delta(x)-\tau\bigl(\Delta(x)\bigr)\Bigr)
\overset{\eqref{eq:bplbi2}}{=} -\,\tau\bigl(\delta(x)\bigr).
\end{align*}
Second, $\delta$ satisfies co-Jacobi identity.
\begin{align*}
(\mathrm{id}_A +\sigma+\sigma^{2})(\mathrm{id}_A \otimes\delta)\delta(x)
&\overset{\eqref{eq:bplbi2}}{=}(\mathrm{id}_A +\sigma+\sigma^{2})(\mathrm{id}_A \otimes\delta)
\bigl(x_{(1)}\otimes x_{(2)}-x_{(2)}\otimes x_{(1)}\bigr) \\
&\overset{\eqref{eq:bplbi2}}{=}(\mathrm{id}_A +\sigma+\sigma^{2})
\Bigl(x_{(1)}\otimes x_{(2)(1)}\otimes x_{(2)(2)}
- x_{(1)}\otimes x_{(2)(2)}\otimes x_{(2)(1)}\\
& \qquad
- x_{(2)}\otimes x_{(1)(1)}\otimes x_{(1)(2)}
+ x_{(2)}\otimes x_{(1)(2)}\otimes x_{(1)(1)} \Bigr)\\
&=
x_{(1)}\otimes x_{(2)(1)}\otimes x_{(2)(2)}
- x_{(1)}\otimes x_{(2)(2)}\otimes x_{(2)(1)}
- x_{(2)}\otimes x_{(1)(1)}\otimes x_{(1)(2)}\\
&\quad
+ x_{(2)}\otimes x_{(1)(2)}\otimes x_{(1)(1)}
+ x_{(2)(2)}\otimes x_{(1)}\otimes x_{(2)(1)}
- x_{(2)(1)}\otimes x_{(1)}\otimes x_{(2)(2)}\\
&\quad
- x_{(1)(2)}\otimes x_{(2)}\otimes x_{(1)(1)}
+ x_{(1)(1)}\otimes x_{(2)}\otimes x_{(1)(2)}
+ x_{(2)(1)}\otimes x_{(2)(2)}\otimes x_{(1)}\\
&\quad
- x_{(2)(2)}\otimes x_{(2)(1)}\otimes x_{(1)}
- x_{(1)(1)}\otimes x_{(1)(2)}\otimes x_{(2)}
+ x_{(1)(2)}\otimes x_{(1)(1)}\otimes x_{(2)}\\
&\overset{\eqref{eq:preLieco}}{=} 0 .
\end{align*}
Finally, $S$ is an averaging operator on $(A, \delta)$.
\begin{align*}
(S\otimes S)\,\delta(x)
&\overset{\eqref{eq:bplbi2}}{=} (S\otimes S)\bigl(\Delta(x)-\tau(\Delta(x))\bigr)\\
&= (S\otimes S)\Delta(x)\;-\;\tau\bigl((S\otimes S)\Delta(x)\bigr)\\
&\overset{\eqref{eq:apLieco}}{=} (S\otimes \mathrm{id}_A)\Delta\bigl(S(x)\bigr)\;-\;\tau\bigl((\mathrm{id}_A\otimes S)\Delta(S(x))\bigr)\\
&= (S\otimes \mathrm{id}_A)\Bigl(\Delta(S(x))-\tau(\Delta(S(x)))\Bigr)\\
&\overset{\eqref{eq:bplbi2}}{=} (S\otimes \mathrm{id}_A)\,\delta\bigl(S(x)\bigr).
\end{align*}
Similarly, $(S\otimes S)\,\delta(x) = (\mathrm{id}_A \otimes S)\,\delta\bigl(S(x)\bigr)$.

\textbf{Step 3.}
By Theorem ~\ref{Theorem:Liebi}, $(A,[\cdot,\cdot],\delta)$ is a Lie bialgebra if and only if $(A,\circ,\Delta)$ is balanced.

\textbf{Step 4.}
We prove that compatibility conditions Eq.~\eqref{eq:aLbi1} and Eq.~\eqref{eq:aLbi2} hold.
\begin{align*}
[P(x),\, S(y)]
\overset{\eqref{eq:bplbi1}}{=} P(x)\circ S(y) - S(y)\circ P(x)
\overset{\eqref{eq:regu1} \eqref{eq:regu2}}{=} S\bigl(P(x)\circ y\bigr) - S\bigl(y\circ P(x)\bigr)
\overset{\eqref{eq:bplbi1}}{=} S\bigl([P(x),\, y]\bigr),
\end{align*}
\begin{align*}
[P(x),\, S(y)]
\overset{\eqref{eq:bplbi1}}{=} P(x)\circ S(y) - S(y)\circ P(x)
\overset{\eqref{eq:regu1} \eqref{eq:regu2}}{=} S\bigl(x\circ S(y) - S(y)\circ x\bigr)
\overset{\eqref{eq:bplbi1}}{=} S\bigl([x,\, S(y)]\bigr).
\end{align*}
\begin{align*}
(S\otimes P)\,\delta(x)
&\overset{\eqref{eq:bplbi1}}{=} (S\otimes P)\Delta(x) - (S\otimes P)\,\tau(\Delta(x)) \\
&= (S\otimes P)\Delta(x) - \tau\bigl((P\otimes S)\Delta(x)\bigr) \\
&\overset{\eqref{eq:aplb2}}{=} (S\otimes \mathrm{id}_A)\Delta\bigl(P(x)\bigr)
   - \tau\bigl((\mathrm{id}_A\otimes S)\Delta(P(x))\bigr) \\
&= (S\otimes \mathrm{id}_A)\Delta\bigl(P(x)\bigr)
   - (S\otimes \mathrm{id}_A)\tau\bigl(\Delta(P(x))\bigr) \\
&\overset{\eqref{eq:bplbi1}}{=} (S\otimes \mathrm{id}_A)\,\delta\bigl(P(x)\bigr).
\end{align*}
Similarly, $(S\otimes P)\,\delta(x) = (\mathrm{id}_A\otimes P)\,\delta\bigl(P(x)\bigr)$.
Then, we have completed this proof.
\end{proof}
\end{theorem}

\medskip
\noindent
\textbf{Further problems.}
As is well known, there is a relationship among dendriform algebras, associative algebras, Lie algebras, and pre-Lie algebras as follows.
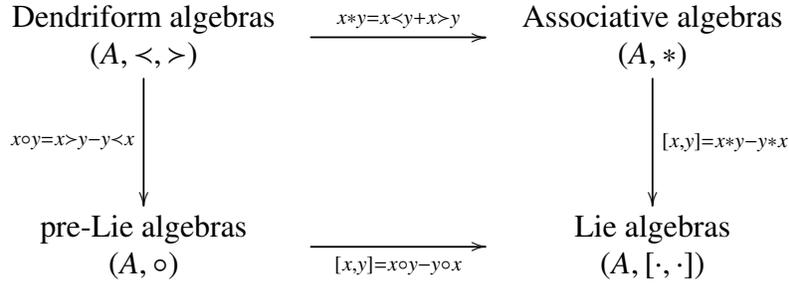
\[
\xymatrix@C=5.5em@R=4.0em{
\parbox{4.2cm}{\centering
Dendriform algebras\\
$(A,\prec,\succ)$}
  \ar[r]^{x*y=x\prec y+x\succ y}
  \ar[d]_{x\circ y=x\succ y-y\prec x}
&
\parbox{4.2cm}{\centering
Associative algebras\\
$(A,*)$}
  \ar[d]^{[x,y]=x*y-y*x}
\\
\parbox{4.2cm}{\centering
pre-Lie algebras\\
$(A,\circ)$}
  \ar[r]_{[x,y]=x\circ y-y\circ x}
&
\parbox{4.2cm}{\centering
Lie algebras\\
$(A,[\cdot,\cdot])$}
}
\]
Motivated by this, it is natural to ask whether an analogous commutative diagram also holds at the level of averaging structures, that is, whether the following diagram is commutative.
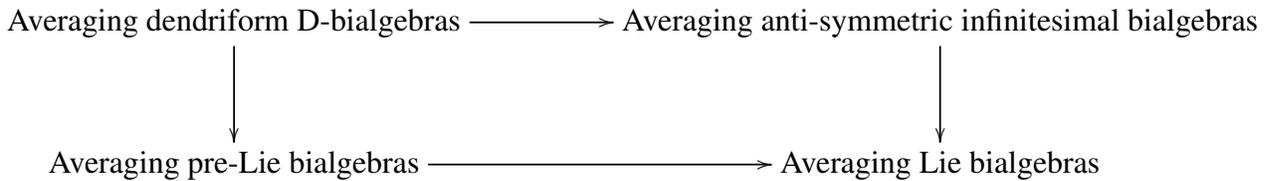
\[
\xymatrix@C=4.5em@R=3.0em{
\text{Averaging dendriform D-bialgebras}
  \ar[r]
  \ar[d]
&
\text{Averaging anti-symmetric infinitesimal bialgebras}
  \ar[d]
\\
\text{Averaging pre-Lie bialgebras}
  \ar[r]
&
\text{Averaging Lie bialgebras}
}
\]
This problem is meaningful and worth investigating. However, due to space limitations, we do not provide a detailed proof here and leave a thorough verification for future work.

\vspace{0.5cm}
\smallskip
\noindent
{\bf Acknowledgments.} This work is supported by Natural Science Foundation of China (Grant No. 12101183). Y. Y. Zhang is also supported by the Postdoctoral Fellowship Program of CPSF under Grant Number (GZC20240406), the Henan Provincial Selective Research Funding Program for Returned Scholars Studying Abroad (HNLX202613) and Sponsored by Natural Science Foundation of Henan (262300421229).

\smallskip
\noindent
{\bf Data availability} All datasets underlying the conclusions of the paper are available to readers.

\smallskip
\noindent
{\bf Conflict of interest} The authors declare that there is no conflict of interest.

\end{document}